\documentclass[a4paper,reqno,final]{amsart}

\usepackage{amsmath,amssymb,amsthm}
\usepackage[foot]{amsaddr}
\usepackage{cite}
\usepackage{xcolor}
\usepackage{nicefrac}
\usepackage{bm}
\title[$\fspace{L}^\infty$-stability of dG time stepping]{The variable-order discontinuous Galerkin time stepping scheme for parabolic evolution problems is uniformly $\fspace{L}^\infty$-stable}
\author[L.~Schmutz and T.~P.~Wihler]{Lars Schmutz \and Thomas P.~Wihler}
\address{Mathematisches Institut, Universit\"at Bern, Sidlerstr.~5, CH-3012 Bern, Switzerland}
\email{lars.schmutz@math.unibe.ch \and wihler@math.unibe.ch}

\thanks{The authors acknowledge the support of the Swiss National Science Foundation (SNF), Grant No.~200021\underline{\phantom{--}}162990}

\newcommand{\set}[1]{\mathcal{#1}}
\renewcommand{\hspace}[1]{\mathbb{#1}}
\newcommand{\hspacedual}[1]{\mathbb{#1}^\star}
\newcommand{\fspace}[1]{\mathrm{#1}}
\newcommand{\operator}[1]{\mathsf{#1}}

\newcommand{\jump}[1]{\lbrack\!\lbrack#1\rbrack\!\rbrack} 
\newcommand{\abs}[1]{\left\lvert#1\right\rvert} 
\newcommand{\NN}[1]{\left\|#1\right\|}
\newcommand{\dualprod}[3]{\left<#1,#2\right>_{\hspacedual{#3}\times\hspace{#3}}}
\newcommand{\dd}{\,\operator{d}}
\newcommand{\Lspace}[2]{\fspace{L}^{#1}(#2)}
\newcommand{\Wspace}[2]{\fspace{W}^{#1}(#2)}
\newcommand{\Cspace}[1]{\fspace{C}^{0}(#1)}
\newcommand{\Lnorm}[3]{\NN{#1}_{\Lspace{#2}{#3}}}
\newcommand{\hnorm}[2]{\NN{#1}_{\hspace{#2}}}
\newcommand{\iprod}[3]{\left(#1,#2\right)_{#3}}
\newcommand{\lifting}[2]{\operator{L}_{#2}^{#1}}
\newcommand{\liftinge}[2]{\widetilde{\operator{L}}_{#2}^{#1}}
\newcommand{\pol}{\hspace{P}^{r_{m}}(I_{m};\hspace{X}_{m})}
\newcommand{\polr}{\hspace{P}^{r_{m}}(I_{m})}
\newcommand{\polrz}{\hspace{P}_{0}^{r_{m}}(I_{m})}
\renewcommand{\wr}{\hspace{W}^{r_{m}}_{\lambda}(I_m)}

\newcommand{\ex}[1]{e^{#1}}
\newcommand{\der}[1]{\frac{\dd}{\dd #1}}
\newcommand{\hk}[1]{K^m_{#1}}

\newcommand{\etam}{\eta_{\lambda,m}^{r_m}}
\newcommand{\projts}{\Pi^{r_{m}}_{m}}
\newcommand{\projtsi}{\Pi^{r_{i}}_{i}}
\newcommand{\gaml}[1]{\Gamma_{\lambda,m}^{#1}}
\newcommand{\gamlindex}[1]{\Gamma_{\lambda_i,m}^{#1}}
\newcommand{\gam}[1]{\Gamma_{m}^{#1}}
\newcommand{\gami}{(\gam{r_m})^{-1}}
\newcommand{\gamii}{(\Gamma_{i}^{r_i})^{-1}}
\newcommand{\gamli}{(\gaml{r_m})^{-1}}
\newcommand{\gamliindex}[1]{(\Gamma_{\lambda_i,m}^{#1})^{-1}}
\newcommand{\phit}[2]{\phi_{#2}^{#1}}

\newcommand{\psit}[2]{\psi_{#2}^{#1}}

\newcommand{\rhoo}{\rho_{m}^{r_m}}
\newcommand{\errorode}[2]{\operator{e}_{#2,#1}}
\newcommand{\ez}{\errorode{m}{\lambda}^+}
\newcommand{\sqrtb}[1]{\left(#1\right)^{\nicefrac12}}
\newcommand{\Cstab}{\Upsilon_{\ref{eq:Cstab}}(r_m, k_m\lambda)}
\newcommand{\Cstabdef}{\Upsilon_{\ref{eq:Cstab}}(r, \varrho)}
\newcommand{\hsol}{\bm{\Psi}^{r_m}}
\newcommand{\hsoli}{\bm{\Psi}^{r_i}}

\DeclareMathOperator{\spn}{span}
\DeclareMathOperator{\sgn}{sign}

\newtheorem{Remark}[equation]{Remark}
\newenvironment{remark}{\begin{Remark}\rm}{\end{Remark}}
\newtheorem{theorem}[equation]{Theorem}
\newtheorem{proposition}[equation]{Proposition}
\newtheorem{corollary}[equation]{Corollary}
\newtheorem{lemma}[equation]{Lemma}

\numberwithin{equation}{section}

\begin{document}

\begin{abstract}
In this paper we investigate the $\fspace{L}^\infty$-stability of fully discrete approximations of abstract linear parabolic partial differential equations. The method under consideration is based on an $hp$-type discontinuous Galerkin time stepping scheme in combination with general conforming Galerkin discretizations in space. Our main result shows that the global-in-time maximum norm of the discrete solution is bounded by the data of the PDE, with a constant that is robust with respect to the discretization parameters (in particular, it is uniformly bounded with respect to the local time steps and approximation orders).
\end{abstract}

\keywords{Discontinuous Galerkin time stepping, Galerkin discretizations, parabolic evolution problems, stability, $hp$-methods}

\subjclass[2010]{65J08,65M12,65M60,65M70}

\maketitle

\section{Introduction}

Let~$\hspace{H}$ and~$\hspace{X}$ be two (real) Hilbert spaces, equipped with the inner products~$\iprod{\cdot}{\cdot}{\hspace{H}}$ and~$\iprod{\cdot}{\cdot}{\hspace{X}}$, respectively, as well as with the corresponding induced norms~$\hnorm{\cdot}{\hspace{H}}$ and~$\hnorm{\cdot}{\hspace{X}}$. The respective dual spaces are denoted by~$\hspacedual{H}$ and~$\hspacedual{X}$. Suppose that $\hspace{X}$ is densely embedded in $\hspace{H}$, and consider the Gelfand triple
\begin{equation}\label{eq:Gelfand}
\hspace{X} \hookrightarrow \hspace{H} \cong \hspacedual{H} \hookrightarrow \hspacedual{X}.
\end{equation}
In this paper, based on a variable-order discontinuous Galerkin (dG) time stepping method in conjunction with a conforming Galerkin approximation in space, we will study the stability of the fully discrete numerical approximation of the linear parabolic problem
\begin{equation}\label{eq:slPDE}
\begin{split}
u'(t)+\operator{A}u(t)&=f(t),\qquad t\in(0,T],\\
u(0) &=u_0.
\end{split}
\end{equation}
Here, $\operator{A}:\,\hspace{X}\to\hspacedual{X}$ is a linear, self-adjoint and time-independent elliptic operator that is coercive and bounded in the sense that there are two constants~$\alpha_{\ref{eq:Aprop}}, \beta_{\ref{eq:Aprop}}>0$ such that
\begin{equation}\label{eq:Aprop}
\begin{split}
\dualprod{\operator{A}v}{v}{X}&\ge\alpha_{\ref{eq:Aprop}}\hnorm{v}{X}^2\qquad\forall v\in\hspace{X},\\
\abs{\dualprod{\operator{A}v}{w}{X}}&\le\beta_{\ref{eq:Aprop}}\hnorm{v}{X}\hnorm{w}{X}\qquad\forall v,w\in\hspace{X}.
\end{split}
\end{equation}
Furthermore we let $f \in \Lspace{2}{(0,T);\hspace{H}}$ and~$u_0\in\hspace{H}$ be a given source term and prescribed initial value, respectively. Applying standard notation for Sobolev and Bochner spaces (cf., e.g., \cite[\S1.5]{Roubicek:05}), a classical weak formulation of~\eqref{eq:slPDE} is to find $u \in \Lspace{2}{(0,T);\hspace{X}} \cap \Wspace{1,2}{(0,T);\hspacedual{X}}$ such that, for every $v \in \hspace{X}$, it holds that
\begin{equation}\label{eq:weakslPDE}
\begin{split}
\dualprod{u'}{v}{X}+\dualprod{\operator{A}u}{v}{X}&=\iprod{f(t)}{v}{\hspace{H}},\qquad t\in(0,T],\\
u(0) &=u_0.
\end{split}
\end{equation}
Here, we signify the duality pairing in $\hspacedual{X} \times \hspace{X}$ by $\dualprod{u}{v}{X}$; incidentally, this dual product can be seen as an extension of the inner product in~$\hspace{H}$, that is, for any $u \in \hspace{H}$, $v \in \hspace{X}$, we have $\iprod{u}{v}{H}=\dualprod{u}{v}{X}$; see, e.g., \cite[\S7.2]{Roubicek:05}. Recalling the continuous embedding
 \begin{equation*}
 \Lspace{2}{0,T;\hspace{X}} \cap \Wspace{1,2}{0,T;\hspacedual{X}} \hookrightarrow \Cspace{0,T;\mathbb{H}},
 \end{equation*}
cf., e.g., \cite[Lemma~7.3]{Roubicek:05}, we conclude that the solution of~\eqref{eq:weakslPDE} is continuous in time, i.e., $u \in \Cspace{0,T;\mathbb{H}}$. Furthermore it holds the stability estimate
\begin{equation}\label{eq:PDEstab}
\NN{u}_{\Lspace{2}{I;\hspace{X}}} + \NN{u'}_{\Lspace{2}{I;\hspacedual{X}}} + \NN{u}_{\Cspace{0,T;\mathbb{H}}} \leq C\left(\NN{u_0}_{\hspace{H}}+\NN{f}_{\Lspace{2}{I;\hspace{H}}} \right);
\end{equation}
see, e.g., \cite[Theorem~8.9]{Roubicek:05}.

In the context of parabolic partial differential equations (PDE), the discontinuous Galerkin time stepping methodology has been introduced a few decades ago in~\cite{Jamet:78}. Since then a lot of research has been conducted on this subject: we point to the classical works~\cite{EJ-I,EJ-II,EJ-IV,EJ-V,EJT85,LarssonThomeeWahlbin:98,Thomee:06}, as well as to the more recent articles~\cite{AkrivisMakridakis:04,AkrivisMakridakisNochetto:09,AkrivisMakridakisNochetto:11,MakridakisNochetto:03,MakridakisNochetto:06,LakkisMakridakis:06}, where a novel reconstruction technique for the purpose of \emph{a posteriori} error estimation has been proposed and analyzed. Whilst these articles mainly focus on low-order temporal Galerkin discretizations of fixed degree, the use of $hp$-type dG methods was proposed in~\cite{SchotzauSchwab:00,SchotzauSchwab:00a}. The $hp$-framework permits to employ locally different time step sizes and arbitrary variations of the local approximation orders, and, thereby, to attain high algebraic or even exponential rates of convergence in time. This feature is particularly powerful if local singularities (for instance, in form of a parabolic time layer due to incompatible initial data) appear~\cite{SchotzauSchwab:00,SchotzauSchwab:01,WerderGerdesSchotzauSchwab:01}, or if highly nonlocal~\cite{MatacheSchwabWihler:05,MatacheSchwabWihler:06} or high-dimensional~\cite{PetersdorffSchwab:04} problems need to be solved.

The present paper centers on the stability of fully discrete $hp$-version dG time discretizations of abstract linear parabolic problems. More precisely, given the solution, $u$, of~\eqref{eq:slPDE}, and its $hp$-dG approximation, $U$, our goal is to argue that the stability estimate~\eqref{eq:PDEstab} holds true also on the discrete level. Indeed, using standard energy arguments, it is fairly straightforward to show that~$U$ is bounded with respect to the~$\fspace{L}^2(\hspace{X})$-norm; indeed, this essentially follows from~\cite[Eq.~(2.18)]{SchotzauSchwab:00} and the boundedness of the duality pairing. In addition, applying a suitable reconstruction~$\widehat U$ of~$U$, see, e.g., \cite[\S2.1]{MakridakisNochetto:06} or~\cite[\S3.6]{GeorgoulisLakkisWihler:17}, and applying an inf-sup stability result (cf., e.g., ~\cite{ESV:16}) shows that~$\widehat U'$ is also stable in the~$\fspace{L}^2(\hspacedual{X})$-norm. 

In the current work our goal is to establish the stability of the discrete solution~$U$ with respect to the $\fspace{L}^\infty(\hspace{H})$-norm. We particularly emphasize on deriving an estimate with a (known) constant~$C>0$ that is \emph{uniformly bounded} with respect to the discretization parameters (i.e., in particular, the local time step lengths and approximation orders). Since our focus is on a pointwise bound, energy arguments are typically not appropriate in the discrete context; indeed, this is due to the fact that suitable test functions (such as cut-off functions) do typically \emph{not} belong to the underlying discrete test space. Furthermore, the application of inverse estimates usually involves constants that scale sub-optimally with respect to the local approximation orders, and, thereby, lead to non-uniform stability results. For these reasons we will pursue a completely different and novel approach: More precisely, we will first derive a pointwise formulation of the fully discrete scheme (Section~\ref{sc:strongform}) using a lifting operator technique as in~\cite{SchotzauWihler:10}; cf. also the temporal reconstruction approach~\cite{MakridakisNochetto:06,ESV:16,GeorgoulisLakkisWihler:17}. Then, we analyze the fully discrete parabolic operator, and  show that its inverse operator is $\fspace{L}^\infty(\hspace{H})$-stable (Section~\ref{sc:abstract}). In order to proceed in this direction, in Section~\ref{sec:introduction:dG}, we will first look at the special case where~$\hspace{H}=\hspace{X}=\mathbb{R}$ in~\eqref{eq:Gelfand}, and construct a representation formula (Section~\ref{sc:repform}) which is composed of two terms: The first term is based on the concept of a dG fundamental solution (Section~\ref{sc:fundsol}), and relates to the initial value, $u_0$, in~\eqref{eq:slPDE}. The second term, analogously as in the classical Duhamel principle, is an integral that involves the product of the right-hand side function, $f$, in~\eqref{eq:slPDE}, and an exponentially decaying expression in time. Subsequently, using a spectral decomposition, we will employ the scalar analysis on each time step in order to derive a stability bound for the inverse parabolic operator in the abstract case (Proposition~\ref{prop:gammapinverse}). Finally, inverting the pointwise form of the dG scheme, and applying the previous stability analysis, eventually implies the main result (Theorem~\ref{thm:main}).

\section{Fully discrete discontinuous Galerkin time stepping}\label{sec:introduction:dG}

\subsection{Variable-order time partitions and discrete spaces}
On an interval~$I=[0,T]$, $T>0$, consider time nodes $0 = t_0 < t_1 < \cdots < t_{M-1} < t_M = T$, which introduce a time partition $\set{M}=\{I_m\}_{m=0}^M$ of~$I$ into~$M+1$ time intervals~$I_m=(t_{m-1},t_m]$, $m=1,\ldots,M$, and~$I_0=\{t_0\}$. The (possibly varying) length $k_m = t_m - t_{m-1}$ of a time interval is called the $m$-th time step. We define the one-sided limits of an $\set{M}$-wise continuous function $v$ at each time node $t_m$, $0\le m\le M-1$, by
\[
v^+_m := \lim_{s\searrow 0} v(t_m+s), \qquad
v^-_m := \lim_{s\searrow 0} v(t_m-s),
\]
where~$v^-_0$ is considered to be a prescribed initial value. Then, the discontinuity jump of $v$ at $t_m$, $0\le m \le M-1$, is defined by $\jump{v}_m := v^+_m - v^-_m$.

Furthermore, to each interval we associate a polynomial degree~$r_m\ge 0$, which takes the role of a local approximation order. Moreover, given any (real) Hilbert (sub)space~$\hspace{V}\subset \hspace{H}$, an integer~$r\in\mathbb{N}_0$, and an interval~$J\subset\mathbb{R}$, the set
\[
\hspace{P}^{r}(J;\hspace{V})=\left\{p\in \Cspace{\bar J;\hspace{V}}:\,p(t)=\sum_{i=0}^rv_it^i,\, v_i\in \hspace{V}\right\}
\]
signifies the space of all polynomials of degree at most~$r$ on~$J$ with values in~$\hspace{V}$. If~$\hspace{V}=\mathbb{R}$, then we simply write~$\polr$.

A fully discrete framework for~\eqref{eq:weakslPDE} is based on replacing the Hilbert space~$\hspace{X}$ from~\eqref{eq:Gelfand} by finite-dimensional subspaces~$\hspace{X}_m\subset \hspace{X}$, $n_m:=\dim(\hspace{X}_m)<\infty$, on each interval~$I_m$, $0\le m\le M$. The $\hspace{H}$-orthogonal projection from~$\hspace{H}$ to~$\hspace{X}_m$, for~$0\le m\le M$, is given by
\begin{equation*}
\pi_m:\, \hspace{H}\to \hspace{X}_m,\qquad v\mapsto\pi_mv:\quad (v-\pi_mv,w)_\hspace{H}=0\quad\forall w\in \hspace{X}_m.
\end{equation*}
Notice the obvious stability property
\begin{equation}\label{eq:stabpi}
\|\pi_mv\|_{\hspace{H}}\le\|v\|_{\hspace{H}}\qquad\forall v\in\hspace{H}.
\end{equation}
Moreover, $\operator{A}_m: \hspace{X} \rightarrow \hspace{X}_{m}$ denotes the discretization of~$\operator{A}$ defined by
\begin{equation}\label{eq:Am}
 \iprod{\operator{A}_m u}{v}{\hspace{H}}=\dualprod{\operator{A}u}{v}{X} \qquad \forall v \in \hspace{X}_{m},
\end{equation}
for~$1\le m\le M$. Recalling~\eqref{eq:Aprop}, we observe that~$\operator{A}_m$ is invertible as an operator from~$\hspace{X}_m$ to~$\hspace{X}_m$.

\subsection{Fully discrete dG time stepping}\label{sc:strongform}

Based on the previous definitions, the \emph{fully discrete dG-in-time/conforming-in-space scheme} for~\eqref{eq:slPDE} is given iteratively as follows: Find~$U|_{I_m}\in\pol$ through the weak formulation
\begin{equation}\label{eq:semilinear-dg}
\begin{split}
  \int_{I_m} (U',V)_\hspace{H}  \dd t 
  +  (\jump{U}_{m-1}, V_{m-1}^+)_\hspace{H} &+ \int_{I_m}\dualprod{\operator{A}U}{V}{\hspace{X}} \dd t\\
  & =   \int_{I_m}(f,V)_\hspace{H} \dd t\qquad\forall V\in\pol,
  \end{split}
\end{equation}
for any~$1\le m\le M$. Here, for~$m=1$, we let with
\begin{equation}\label{eq:u0}
U_0^-:=\pi_0u_0,
\end{equation}
where~$u_0\in \hspace{H}$ is the initial value from~\eqref{eq:slPDE}, and, thereby, $\jump{U}_0=U_0^+-\pi_0u_0$. 

In order to write~\eqref{eq:semilinear-dg} in pointwise form, we proceed along the lines of~\cite{SchotzauWihler:10}. Specifically, for $1 \leq m \leq M$, and any~$z\in\hspace{X}_m$, we define the (linear) lifting operator
\begin{equation*}
\lifting{r_m}{m}: \hspace{X}_m \rightarrow \pol
\end{equation*}
by
\begin{equation*}
 \int_{I_{m}}\iprod{\lifting{r_m}{m}(z)}{V}{\hspace{H}} \dd t = \iprod{z}{V(t_{m-1})}{\hspace{H}} \qquad \forall v \in \pol.
\end{equation*}
Referring to~\cite[Lemma~6]{SchotzauWihler:10} there holds the explicit representation formula
\begin{equation}\label{eq:liftingstrong}
\lifting{r_m}{m}(z) = \frac{z}{k_{m}}\sum_{i=0}^{n} (-1)^{i}(2i+1)K^m_{i}(t),
\end{equation}
where $\{K^m_{i}\}_{i\ge 0}$ is the family of Legendre polynomials, affinely scaled from~$[-1,1]$ to $I_{m}$, such that
\begin{equation}\label{eq:Leg1}
(-1)^iK^m_i(t_{m-1})=K^m_i(t_m)=1,\qquad i\ge 0,
\end{equation}
and
\begin{equation}\label{eq:Leg2}
 \int_{I_{m}} K^m_{i}(t) K^m_{j}(t) \dd t = \frac{k_{m}}{2i+1}\delta_{ij} \qquad \forall i,j \in \mathbb{N}_0;
\end{equation}
see~\cite[\S3.1]{SchotzauWihler:10} for details.
For later purposes, we also introduce the endpoint lifting operator
\begin{equation*}
\liftinge{r_m}{m}: \hspace{X}_m \rightarrow \pol
\end{equation*}
by
\begin{equation*}
  \int_{I_{m}}\iprod{\liftinge{r_m}{m}(z)}{V}{\hspace{H}} \dd t = \iprod{z}{V(t_{m})}{\hspace{H}} \qquad \forall v \in \pol.
\end{equation*}
Using~\eqref{eq:liftingstrong} and~\eqref{eq:Leg1}, we may represent it as
\begin{equation}\label{eq:liftingstronge}
\liftinge{r_m}{m}(z) = \frac{z}{k_{m}}\sum_{i=0}^{r_m} (-1)^{i}(2i+1)K^m_{i}(-t) 
= \frac{z}{k_{m}}\sum_{i=0}^{r_m} (2i+1)K^m_{i}(t).
\end{equation}

Let $\projts : \fspace{L}^{2}(I_{m},\hspace{H}) \rightarrow \pol$ denote the fully discrete $\fspace{L}^{2}(I_{m},\hspace{H})$-projec\-tion defined by
\begin{equation*}
 \int_{I_{m}}\iprod{\projts (U)}{V}{\hspace{H}} \dd t = \int_{I_{m}}\iprod{U}{V}{\hspace{H}}\dd t \qquad \forall V \in \pol.
\end{equation*}
Then, employing the spatial projection~$\pi_m$ from~\eqref{eq:stabpi} and the discrete elliptic operator~$\operator{A}_m$ from~\eqref{eq:Am}, and using the lifting operator~$\lifting{r_m}{m}$, we transform~\eqref{eq:semilinear-dg} into
\begin{equation*}
\begin{split}
  \int_{I_m} \iprod{U' + \lifting{r_m}{m}(\pi_m\jump{U}_{m-1}) + \operator{A}_{m}U-\projts f}{V}{\hspace{H}} \dd t = 0 \qquad\forall V\in\pol. 
  \end{split}
\end{equation*}
This immediately implies the pointwise form
\begin{equation}\label{eq:semilinear-dg-strong}
\begin{split}
  U' + \lifting{r_m}{m}(\pi_m\jump{U}_{m-1}) + \operator{A}_{m}U = \projts f,\qquad t\in I_m.  
  \end{split}
\end{equation}
Following~\cite{HolmWihler:15}, for~$1\le m\le M$, we consider the dG-time operator 
\[
\chi_m^{r_m}:\, \pol \rightarrow \pol, 
\]
given by
\begin{equation}\label{eq:chim}
 \chi_m^{r_m} (U) := U' + \lifting{r_m}{m}(U_{m-1}^{+}),\qquad U\in\pol.
\end{equation}
Consequently, introducing the operator
\[
\gam{r_m}:\, \pol \rightarrow \pol 
\]
by
\begin{equation}\label{eq:gam}
 \gam{r_m}: = \chi_m^{r_m}+\operator{A}_{m},
\end{equation}
we can write~\eqref{eq:semilinear-dg-strong} as
\begin{equation}\label{eq:strong2}
 \gam{r_{m}}(U)= \projts f + \lifting{r_m}{m}(\pi_mU_{m-1}^{-}),
\end{equation}
for~$1\le m\le M$. Referring to~\cite[Proposition~2.6]{SchotzauSchwab:00}, we note that~\eqref{eq:semilinear-dg}  is uniquely solvable, and, hence, the operator~$\gam{r_m}$ from~\eqref{eq:gam} is an isomorphism on~$\pol$.

\section{Scalar problem in~$\mathbb{R}$}\label{sec:scalar}

In order to derive a stability analysis for the fully discrete scheme~\eqref{eq:strong2}, we focus first on the case where~$\hspace{H}=\hspace{X}=\mathbb{R}$. Specifically, for $1\le m\le M$, consider the scalar problem of finding a function~$u:\,I_m\to\mathbb{R}$ such that
\begin{equation*}
\begin{split}
u'(t)+\lambda u(t)&= f(t),\qquad t \in I_m,\\
u(t_{m-1}) &= u_{m-1}.
\end{split}
\end{equation*}
Here, $\lambda>0$ is a fixed parameter, $u_{m-1}\in\mathbb{R}$ is a prescribed initial value, and~$f:\,[0,T]\to\mathbb{R}$ is a given source function. The dG time discretization of this problem is formulated in strong form as
\begin{equation}\label{eq:scalar}
\gaml{r_m}(U) = \projts f+\lifting{r_m}{m}(u_{m-1}),\qquad t\in I_m,
\end{equation}
where, in this simplified context,~$\projts:\,\fspace{L}^2(I_m)\to\polr$ is the $\fspace{L}^2$-projection onto~$\polr$, and
\begin{equation}\label{eq:gamscalar}
\gaml{r_m} : \polr \rightarrow \polr,\qquad
 \gaml{r_m}(v)= \chi_m^{r_m}(v) + \lambda v,
\end{equation}
is the scalar version of~\eqref{eq:gam}. As mentioned earlier~$\gaml{r_m}$ is an isomorphism on~$\polr$. Hence, applying the inverse operator~$(\gaml{r_m})^{-1}$ to~\eqref{eq:scalar}, the dG solution~$U$ on $I_m$ can be represented as follows:
\begin{align}\label{eq:nhdG}
 U &= \gamli\left[\lifting{r_m}{m}(u_{m-1})+ \projts{f}\right],\qquad \text{on } I_m.
 \end{align}
Consequently, the stability of the inverse of~$\gam{r_m}$ is crucial in our analysis. We will attend to this matter by means of the classical scalar model problem
\begin{equation}\label{eq:slODE}
\begin{split}
\psi'(t)+\lambda \psi(t)&= 0,\qquad t \in I_m,\\
\psi(t_{m-1}) &= 1,
\end{split}
\end{equation}
with the solution~$\psi(t)=e^{-\lambda(t-t_{m-1})}$.

\subsection{DG fundamental solution}\label{sc:fundsol}
We denote the dG time stepping approximation of~\eqref{eq:slODE} by $\psit{r_m}{\lambda}\in\polr$, and call it the \emph{dG fundamental solution of degree~$r_m$ on~$I_m$}. Based on~\eqref{eq:scalar} and~\eqref{eq:nhdG}, with~$f\equiv 0$, and~$u_{m-1}=1$, it holds that
\begin{equation}\label{eq:psi-characterization}
 \gaml{r_m}(\psit{r_m}{\lambda}) = \lifting{r_m}{m}(1),
\end{equation}
and
\begin{equation}\label{eq:psi-characterization2}
 \psit{r_m}{\lambda} = \gamli (\lifting{r_m}{m}(1)),
\end{equation}
respectively.

Our goal is to derive an explicit representation formula for~$(\gaml{r_m})^{-1}$. To this end, we consider the subspace
\begin{equation*}
 \polrz := \{ v \in \polr : v(t_{m-1}) = 0 \},
\end{equation*}
as well as its image under $\gaml{r_m}$, i.e.,
\begin{equation*}
 \wr := \gaml{r_m}(\polrz).
\end{equation*}

\begin{lemma}\label{lem:wr-plus-lifting}
Let~$\lambda\ge0$. There holds $\dim \wr = r_m$, and we have the direct sum
 \begin{equation*}
  \wr \oplus \spn \{\lifting{r_m}{m}(1)\} = \polr.
 \end{equation*}
\end{lemma}

\begin{proof}
If~$\lambda=0$, then the result simply follows by observing that the derivative operator maps the space~$\polr$ onto~$\hspace{P}^{r_m-1}(I_m)$ if~$r_m>0$, and by noticing that the lifting operator is of exact degree~$r_m$. Hence, let us consider the case~$\lambda>0$. For~$i\ge 0$, and $1\le m\le M$, consider the integrated Legendre polynomials 
\begin{equation}\label{eq:intLeg}
Q^m_{i}(t):= \frac{2}{k_m}\int_{t_{m-1}}^{t}K^m_{i}(s)\dd s,\qquad t\in I_m.
\end{equation}
Evidently, the set~$\{Q^m_{i} \}_{i=0}^{r_{m}-1}$ is a basis of $\polrz$. Furthermore, since the polynomial degree of $\gaml{r_m} (Q^m_{i})$ is exactly~$i$, for~$i\ge 0$, it follows that $\{ \gaml{r_m}(Q^m_i) \}_{i=0}^{r_{m}-1}$ forms a basis of $\wr$. It therefore remains to show that the intersection of $\spn\{\lifting{r_m}{m}(1)\}$ and~$\wr$ is trivial. Take any~$w\in\wr$, and choose~$v\in\polrz$ such that~$w=v'+\lambda v =\alpha\lifting{r_m}{m}(1)$, for some~$\alpha\in\mathbb{R}$. Then, testing by~$v$, and integrating over $I_m$, yields
\[
0 = \alpha v_{m-1}^+=\alpha\int_{I_m}\lifting{r_m}{m}(1)v\dd t=\int_{I_m}(v'+ \lambda v)v \dd t = \frac{1}{2} v(t_m)^2 + \lambda \Lnorm{v}{2}{I_{m}}^2.
\]
Hence, we conclude that~$v\equiv 0$, and, therefore~$w\equiv 0$.
\end{proof}

It is interesting and useful for the subsequent analysis to notice that the above setup gives rise to the \emph{dG dual solution of degree~$r_m$ on~$I_m$}, which we denote by~$\phit{r_m}{\lambda}\in\polr$. It is defined via the differential equation
\begin{equation}\label{eq:phit-characterization}
 (\phit{r_m}{\lambda})' - \lambda \phit{r_m}{\lambda} - \liftinge{r_m}{m}(\phit{r_m}{\lambda} (t_m)) = \lifting{r_m}{m}(1),
\end{equation}
where the lifting operators~$\lifting{r_m}{m}$ and~$\liftinge{r_m}{m}$ are given in~\eqref{eq:liftingstrong} and~\eqref{eq:liftingstronge}, respectively, with~$\hspace{X}_m$ being replaced by~$\mathbb{R}$.

\begin{lemma}\label{lem:phit-characterization}
Suppose that~$\lambda\ge 0$. There exists exactly one solution of~\eqref{eq:phit-characterization} in~$\polr$, i.e., the dG dual solution~$\phit{r_m}{\lambda}$ is well-defined in~$\polr$. Furthermore, $\phit{r_m}{\lambda}$  is~$\fspace{L}^2$-orthogonal to~$\wr$, i.e.,
\begin{equation*}
  \int_{I_{m}} \phit{r_m}{\lambda} w \dd t = 0 \qquad \forall w \in \wr.
\end{equation*}
\end{lemma}

\begin{proof}
Let us define an operator
$\operator{\Psi}: \polr \rightarrow \polr$ by
\begin{equation}\label{eq:Psi-operator}
 \operator{\Psi}(v) := v' - \lambda v - \liftinge{r_m}{m}(v (t_m)).
\end{equation}
We show that the kernel of~$\operator{\Psi}$ is trivial, i.e., $\operator{\Psi}$ is an isomorphism. Suppose that~$v\in\polr$, and~$\operator{\Psi}(v)\equiv 0$. In case that~$\lambda=0$, this implies that~$v'=\liftinge{r_m}{m}(v (t_m))$. Now, since~$\liftinge{r_m}{m}(v (t_m))$ has degree exactly~$r_m$, unless~$v (t_m)=0$, we conclude that~$v'\equiv0$ as well as~$v(t_m)=0$. This, in turn, leads to~$v\equiv 0$. Otherwise, if~$\lambda>0$,  we test~\eqref{eq:Psi-operator} by~$v\in\polr$, and integrate over~$I_m$. Then,
\begin{align*}
 0=\int_{I_m}\operator{\Psi}(v)v\dd t 
 &=  \int_{I_m}\left(v' - \lambda v - \liftinge{r_m}{m}(v (t_m))\right)v\dd t \\
	    &= \frac{1}{2}\left(v(t_m)^2 - v(t_{m-1})^2\right) - \lambda \Lnorm{v}{2}{I_m}^{2} - v(t_m)^2 \\
	    &= -\frac{1}{2}(v(t_m)^2 + v(t_{m-1})^2) - \lambda \Lnorm{v}{2}{I_m}^{2}.
\end{align*}
This immediately results in~$v\equiv 0$. Hence, there exists exactly one
$\phit{r_m}{\lambda} \in \polr$ such that $\operator{\Psi}(\phit{r_m}{\lambda}) = \lifting{r_m}{m}(1)$.

In order to prove the second assertion, we let $w \in \wr$, and choose $v \in \polrz$ such that $w = v' + \lambda v$. Then, integrating by parts, there holds that
\begin{align*}
 \int_{I_m} \phit{r_m}{\lambda} w \dd t &= \int_{I_m} \phit{r_m}{\lambda} (v ' + \lambda v) \dd t \\
   &= \int_{I_m} (-(\phit{r_m}{\lambda})' + \lambda \phit{r_m}{\lambda})v \dd t + \phit{r_m}{\lambda} (t_m) v(t_m) \\
   &=  \int_{I_m} (-(\phit{r_m}{\lambda})' + \lambda \phit{r_m}{\lambda} + \liftinge{r_m}{m}(\phit{r_m}{\lambda}(t_m))v \dd t.
\end{align*}
Invoking~\eqref{eq:phit-characterization}, we obtain
\begin{align*}
\int_{I_m} \phit{r_m}{\lambda} w \dd t
&= \int_{I_m} -\lifting{r_m}{m}(1)v 
= -v(t_{m-1}) 
= 0.
\end{align*}
Therefore, $\phit{r_m}{\lambda}$ is in the orthogonal complement of $\wr$.
\end{proof}

\begin{lemma}\label{lem:init-val-psi}
Let~$\lambda\ge 0$. The initial values of the dG fundamental solution~$\psit{r_m}{\lambda}$ and the dG dual solution~$\phit{r_m}{\lambda}$ satisfy
 \begin{equation}\label{eq:init-val-psi}
  \phit{r_m}{\lambda}(t_{m-1}) = -\psit{r_m}{\lambda}(t_{m-1}).
 \end{equation}
\end{lemma}

\begin{proof}
Testing~\eqref{eq:phit-characterization} by $\psit{r_m}{\lambda}$, and integrating over $I_m$ by parts, we obtain
 \begin{align*}
0 &=  \int_{I_{m}} \left((\phit{r_m}{\lambda})' - \lambda \phit{r_m}{\lambda} - \liftinge{r_m}{m}(\phit{r_m}{\lambda} (t_m)) - \lifting{r_m}{m}(1)\right) \psit{r_m}{\lambda} \dd t \\
 &= - \int_{I_{m}} \phit{r_m}{\lambda} (\psit{r_m}{\lambda})' \dd t+ \phit{r_m}{\lambda}(t_{m})\psit{r_m}{\lambda}(t_{m}) - \phit{r_m}{\lambda}(t_{m-1})\psit{r_m}{\lambda}(t_{m-1}) \\
 &\quad- \lambda \int_{I_m} \phit{r_m}{\lambda} \psit{r_m}{\lambda} \dd t - \phit{r_m}{\lambda}(t_{m})\psit{r_m}{\lambda}(t_{m}) - \psit{r_m}{\lambda}(t_{m-1})  \\
 &= - \int_{I_{m}}\phit{r_m}{\lambda} \left\{ (\psit{r_m}{\lambda})'+ \lambda \psit{r_m}{\lambda} + \lifting{r_m}{m}(\psit{r_m}{\lambda}(t_{m-1})) \right\} \dd t - \psit{r_m}{\lambda}(t_{m-1}).
 \end{align*}
Recalling the definition~\eqref{eq:psi-characterization} of the dG fundamental solution, yields
\begin{align*}
0
 &= - \int_{I_{m}} \phit{r_m}{\lambda}\lifting{r_m}{m}(1) \dd t - \psit{r_m}{\lambda}(t_{m-1}) 
 = - \phit{r_m}{\lambda}(t_{m-1}) - \psit{r_m}{\lambda}(t_{m-1}),
 \end{align*}
which is~\eqref{eq:init-val-psi}.
%
\end{proof}

Our next step is to prove that the dG dual solution takes the value of its maximum norm at~$t_{m-1}$.

\begin{proposition}[Stability of $\phit{r_{m}}{\lambda}$]\label{prop:phitimaxi}
Suppose that~$\lambda>0$. It holds
\begin{equation}\label{eq:phitimax}
\Lnorm{\phit{r_m}{\lambda}}{\infty}{I_m} = \lvert \phit{r_m}{\lambda}(t_{m-1}) \rvert,
\end{equation}
and
\begin{equation}\label{eq:phitisandwich}
-1 < \phit{r_m}{\lambda}(t_{m-1}) < 0.
\end{equation}
\end{proposition}

The proof of the above proposition, to be presented later on, is based on some properties of the Legendre expansion of the dG dual solution. More precisely, write
\begin{equation}\label{eq:phiexp}
\phit{r_m}{\lambda} = \sum_{i=0}^{r_m} a_i K^m_i,
\end{equation}
with the Legendre polynomials~$\{K_i^m\}_{i\ge0}$ from~\eqref{eq:Leg1} and~\eqref{eq:Leg2}.

\begin{lemma}\label{lem:coeff-phit}
Let~$\lambda>0$. Then, for the coefficients $a_0,\ldots,a_{r_m}$ in the Legendre expansion~\eqref{eq:phiexp} there hold the recursion formulas
\begin{align}
 a_{0} &= -\frac{1 + \phit{r_m}{\lambda}(t_{m-1})}{k_m\lambda},\label{eq:a0}\\
 a_{1} &= -\frac{3}{\lambda}\left(\frac{2}{k_m}+\lambda\right) a_0,\label{eq:a1} \\
 a_{i} &= (2i+1)\left(\frac{a_ {i-2}}{2i-3} - \frac{2a_{i-1}}{k_m\lambda}\right), \qquad\text{for } 2 \leq i \leq r_{m}.\label{eq:ai}
 \end{align}
Furthermore, we have that $a_{i} \neq 0$, as well as
\begin{equation}\label{eq:signa}
\sgn(a_i) = (-1)^{i+1},
\end{equation} 
for any $i=0,\ldots,r_{m}$.
\end{lemma}

\begin{proof}
We begin by integrating~\eqref{eq:phit-characterization} over~$I_m$, which yields
\[
\lambda\int_{I_m}\phit{r_m}{\lambda}(t)\dd t=-1-\phit{r_m}{\lambda}(t_{m-1}).
\] 
Then, making use of the expansion~\eqref{eq:phiexp} as well as of the fact that
\[
\int_{I_m}K^m_i\dd t=0\qquad\forall i\ge 1,
\]
we see that
\[
\lambda k_ma_0 = -1-\phit{r_m}{\lambda}(t_{m-1}),
\]
and hence,
\[
a_0= -\frac{1 + \phit{r_m}{\lambda}(t_{m-1})}{k_m\lambda},
\]
which proves~\eqref{eq:a0}. Next, we employ again the integrated Legendre polynomials defined in~\eqref{eq:intLeg}, and notice the following properties, see, e.g., \cite[Eq.~(9)]{SchotzauWihler:10}:
\begin{equation}\label{eq:intLegprop}
\begin{split}
Q^m_0 &= K^m_0 + K^m_1,\qquad
Q^m_i = \frac{1}{2i+1}(K^m_{i+1}-K^m_{i-1}),\quad i\ge 1.
 \end{split}
\end{equation}
Due to Lemma \ref{lem:phit-characterization} we note that
\begin{equation*}
 0 = \int_{I_m}\phit{r_m}{\lambda}(v' + \lambda v)\dd t \qquad \forall v \in \polrz.
\end{equation*}
Thus, applying the expansion~\eqref{eq:phiexp}, and choosing $v:=Q^m_{j}$, we obtain 
\[
0=\sum_{i=0}^{r_m}a_i\int_{I_m}K^m_i\left(\frac{2}{k_m}K^m_j+\lambda Q^m_j\right)\dd t,\qquad j=0,\ldots,r_{m-1}.
\]
Involving~\eqref{eq:intLegprop}, and using the orthogonality property~\eqref{eq:Leg2} of the Legendre polynomials, we arrive at
\begin{equation}\label{eq:reca}
 \begin{split}
  0 &= \left(\frac{2}{k_m}+\lambda\right) a_{0} + \frac{\lambda}{3} a_1, \\
  0 &= -\frac{\lambda}{2j-1}a_{j-1}+\frac{2}{k_m}a_j + \frac{\lambda}{2j+3}a_{j+1}, \qquad 1 \leq j \leq r_{m-1}.
 \end{split}
\end{equation}
Rewriting these equalities yields the asserted recursion relations~\eqref{eq:a1} and~\eqref{eq:ai}. Here, we note that~$a_0\neq 0$ since otherwise all coefficients would be zero, which, in turn, would lead to~$\phit{r_m}{\lambda} \equiv 0$. Moreover, the recursion formulas~\eqref{eq:reca} immediately show that the coefficients~$a_j$, $j=1,\ldots,r_m$, never vanish, and have alternating signs.

It remains to show the sign alternation property~\eqref{eq:signa}. To this end, we test~\eqref{eq:phit-characterization} by the Legendre polynomial~$K^m_{r_m}$, and integrate over $I_m$. Then, observing that $(\phit{r_m}{\lambda})'$ is $\fspace{L}^2$-orthogonal to~$K^m_{r_m}$ (because it has degree~$r_m-1$), and applying the properties~\eqref{eq:Leg1} and~\eqref{eq:Leg2}, leads to
\begin{align*}
 0 
   &= -\frac{k_m\lambda }{2r_m +1}a_{r_m}- \phit{r_m}{\lambda}(t_m) - (-1)^{r_m},
\end{align*}
and therefore,
\begin{equation}\label{eq:arm}
 a_{r_m} = -\frac{2r_m + 1}{k_m \lambda}\left( \phit{r_m}{\lambda}(t_m)+(-1)^{r_m}\right).
\end{equation}
Next, we test \eqref{eq:phit-characterization} by~$\phit{r_m}{\lambda}$, and integrate over~$I_m$. A brief calculation reveals that
\begin{equation}\label{eq:phitioneandmo} 
2 \lambda \Lnorm{\phit{r_m}{\lambda}}{2}{I_m}^2 + \lvert \phit{r_m}{\lambda}(t_m) \rvert^{2} = -\phit{r_m}{\lambda}(t_{m-1})(2 + \phit{r_m}{\lambda}(t_{m-1})).
\end{equation}
Since the left-hand side of \eqref{eq:phitioneandmo} consists only of non-negative terms, it follows that~$\phit{r_m}{\lambda}(t_{m-1}) \in [-2,0]$. In addition, we note that $\max_{x\in[-2,0]}\left[-x(2+x)\right]=1$. Hence, the right-hand side of~\eqref{eq:phitioneandmo}, and thereby also the left-hand side, are both bounded by~1. This implies, in particular, that~$|\phit{r_m}{\lambda}(t_m)|\le1$. Therefore, from~\eqref{eq:arm}, and because~$a_{r_m}\neq 0$, we infer that $\sgn(a_{r_m})=(-1)^{r_m +1}$. Since the sign of the coefficients~$a_j$ are alternating, we necessarily arrive at
\[
\sgn(a_j)=\sgn(a_{r_m})(-1)^{r_m-j}=(-1)^{2r_m+1 -j}=(-1)^{j+1},
\]
for~$0\le j\le r_m$.
\end{proof}

\begin{proof}[Proof of Proposition~\ref{prop:phitimaxi}]
We apply the Legendre expansion~\eqref{eq:phiexp} of~$\phit{r_m}{\lambda}$. Then, recalling~\eqref{eq:signa}, and invoking~\eqref{eq:Leg1}, we deduce that
\begin{equation}\label{eq:phitatmone}
\phit{r_m}{\lambda}(t_{m-1})  = -\sum_{i=0}^{r_m}\lvert a_i \lvert<0.
\end{equation}
This is the upper bound in~\eqref{eq:phitisandwich}. In addition, noticing the fact that \begin{equation}\label{eq:Leg3}
\Lnorm{K_i^m}{\infty}{I_m} = 1,
\end{equation}
we infer
\begin{align}\label{eq:max-phit}
\Lnorm{\phit{r_m}{\lambda}}{\infty}{I_m} \leq \sum_{i=0}^{r_m} \lvert a_{i} \rvert \Lnorm{K^m_i}{\infty}{I_m} = \sum_{i=0}^{r_m} \lvert a_{i} \rvert.
\end{align}
Combining~\eqref{eq:phitatmone} and~\eqref{eq:max-phit}, we arrive at~\eqref{eq:phitimax}. Finally, the lower bound in~\eqref{eq:phitisandwich} follows from the fact that~$a_0<0$, cf.~\eqref{eq:signa}, and from~\eqref{eq:a0}. 
\end{proof}

The ensuing lemma provides further properties of the dG dual solution which will be crucial in the stability analysis below.

\begin{lemma}\label{lem:arm}
For~$\lambda>0$, the coefficient $a_{r_m}$ in the Legendre expansion of~$\phit{r_m}{\lambda}$, cf.~\eqref{eq:phiexp}, satisfies the bound
\[
\lvert a_{r_m} \rvert\le \left(1+\frac{\lambda k_m}{2(2r_m +1)}\right)^{-1}.
\]
\end{lemma}
\begin{proof}
We use the formulas for the Legendre coefficients~$a_0,\ldots a_{r_m}$ of~$\phit{r_m}{\lambda}$ from Lemma~\ref{lem:coeff-phit}. Specifically, from~\eqref{eq:a0} and~\eqref{eq:signa} it follows that
\begin{equation}\label{eq:|a0|}
\lambda k_m|a_0|=1+\phit{r_m}{\lambda}(t_{m-1}).
\end{equation}
Moreover, taking moduli in~\eqref{eq:a1}, we deduce that
\begin{equation}\label{eq:arm-ineq2}
\lvert a_1 \rvert = \frac{3(2 + \lambda k_m )}{\lambda k_m} \lvert a_0 \rvert.
\end{equation}
In addition, rearranging~\eqref{eq:ai}, we have
\[
a_i=\frac{\lambda k_m}{2}\left(\frac{a_{i-1}}{2i-1}-\frac{a_{i+1}}{2i+3}\right),
\qquad 1\le i\le r_m-1,
\]
which, involving again~\eqref{eq:signa}, leads to
 \begin{equation}\label{eq:alt-coeff}
 \lvert a_{i} \rvert = \frac{\lambda k_m}{2}\left(\frac{\lvert a_{i+1} \rvert}{2i+3} - \frac{\lvert a_{i-1} \rvert}{2i-1} \right), \qquad 1 \leq i \leq r_m-1.
 \end{equation}
Inserting~\eqref{eq:alt-coeff} into~\eqref{eq:phitatmone} implies
\begin{align*}
 -\phit{r_m}{\lambda}(t_{m-1}) 
	&= \lvert a_0 \rvert + \lvert a_{r_m} \rvert + \frac{\lambda k_m}{2} \sum_{i=1}^{r_m -1} \left(\frac{\lvert a_{i+1} \rvert}{2i+3} - \frac{\lvert a_{i-1} \rvert}{2i-1} \right).
\end{align*}
Observing the telescope sum on the right-hand side results in
\begin{align*}
 -&\phit{r_m}{\lambda}(t_{m-1}) \\
& =  \left(1-\frac{\lambda k_m}{2}\right)\lvert a_0 \rvert 
 - \frac{\lambda k_m}{6}\lvert a_1 \rvert 
 + \frac{\lambda k_m}{2(2r_m -1)}\lvert a_{r_m - 1} \rvert 
 + \left(1+\frac{\lambda k_m}{2(2r_m +1)}\right)\lvert a_{r_m} \rvert.
\end{align*}
Applying~\eqref{eq:arm-ineq2}, we note that
\[
 -\phit{r_m}{\lambda}(t_{m-1}) 
 =  -\lambda k_m|a_0|
 + \frac{\lambda k_m}{2(2r_m -1)}\lvert a_{r_m - 1} \rvert 
 + \left(1+\frac{\lambda k_m}{2(2r_m +1)}\right)\lvert a_{r_m} \rvert.
\]
Making use of~\eqref{eq:|a0|}, we arrive at
\[
1= \frac{\lambda k_m}{2(2r_m -1)}\lvert a_{r_m - 1} \rvert 
 + \left(1+\frac{\lambda k_m}{2(2r_m +1)}\right)\lvert a_{r_m} \rvert,
 \]
which yields the bound
\[
1\ge \left(1+\frac{\lambda k_m}{2(2r_m +1)}\right)\lvert a_{r_m} \rvert.
\]
This completes the proof.
\end{proof}

\begin{lemma}\label{lem:phiL2}
Let~$\lambda>0$. For the dG dual solution from~\eqref{eq:phit-characterization} there holds
\begin{equation*}
\Lnorm{\phit{r_m}{\lambda}}{2}{I_m}
\le \Cstab k_m^{\nicefrac12},
\end{equation*}
where, for~$r\in\mathbb{N}_0$ and~$\varrho>0$, we let
\begin{equation}\label{eq:Cstab}
\Cstabdef:=\sqrtb{\frac{3}{2(2r + 1) + \varrho}}.
\end{equation}
In particular, $\Lnorm{\phit{r_m}{\lambda}}{2}{I_m}\to0$, as~$r_m\to\infty$, uniformly with respect to~$\lambda$.
\end{lemma}

\begin{proof}
Recalling~\eqref{eq:arm} we have that
\begin{equation}\label{eq:phim}
\left|\phit{r_m}{\lambda}(t_m) - (-1)^{r_m+1}\right|
=\frac{k_m\lambda }{2r_m +1}|a_{r_m}|.
\end{equation}
Moreover, due to Proposition \ref{prop:phitimaxi}, we notice that
\[
0<1+\phit{r_m}{\lambda}(t_{m-1})
=1-\Lnorm{\phit{r_m}{\lambda}}{\infty}{I_m}
\le 1-\left|\phit{r_m}{\lambda}(t_{m})\right|
\le\left|(-1)^{r_m+1}-\phit{r_m}{\lambda}(t_{m})\right|.
\]
Hence,
\begin{equation}\label{eq:phim2}
0<1+\phit{r_m}{\lambda}(t_{m-1})\le\frac{k_m\lambda }{2r_m +1}|a_{r_m}|.
\end{equation}
From~\eqref{eq:phitioneandmo}, we recall that
\begin{equation}\label{eq:phitL2}
2 \lambda \Lnorm{\phit{r_m}{\lambda}}{2}{I_m}^2 =- \lvert \phit{r_m}{\lambda}(t_m) \rvert^{2}  -\phit{r_m}{\lambda}(t_{m-1})(2 + \phit{r_m}{\lambda}(t_{m-1})).
\end{equation}
We estimate the terms on the right-hand side of the above identity separately. Firstly,
\begin{align*}
|\phit{r_m}{\lambda}(t_m)|^2
&=\left|\phit{r_m}{\lambda}(t_m)-(-1)^{r_m+1}\right|^2+2(-1)^{r_m+1}\phit{r_m}{\lambda}(t_m)-1\\
&\ge 1+2(-1)^{r_m+1}\left(-(-1)^{r_m+1}+\phit{r_m}{\lambda}(t_m)\right)\\
&\ge 1-2\left|\phit{r_m}{\lambda}(t_m)-(-1)^{r_m+1}\right|,
\end{align*}
and thus, upon exploiting~\eqref{eq:phim},
\[
-|\phit{r_m}{\lambda}(t_m)|^2\le -1+\frac{2k_m\lambda}{2r_m+1}|a_{r_m}|.
\]
Next, with~\eqref{eq:phim2}, it follows that
\[
2+\phit{r_m}{\lambda}(t_{m-1})\le 1+\frac{k_m\lambda}{2r_m+1}|a_{r_m}|.
\]
Inserting these estimates into~\eqref{eq:phitL2}, and recalling the fact that there holds $0<-\phit{r_m}{\lambda}(t_{m-1})<1$, cf.~\eqref{eq:phitisandwich}, we conclude that
\[
2\lambda\Lnorm{\phit{r_m}{\lambda}}{2}{I_m}^2
\le \frac{3k_m\lambda}{2r_m+1}|a_{r_m}|.
\]
Finally, employing Lemma~\ref{lem:arm}, results in
\[
2\lambda\Lnorm{\phit{r_m}{\lambda}}{2}{I_m}^2
\le \frac{6k_m\lambda}{2(2r_m + 1) + k_m \lambda},
\]
and dividing by~$2\lambda$ completes the proof.
\end{proof}

\begin{remark}
For~$\lambda=0$ it is fairly elementary to verify that
\begin{equation}\label{eq:phit0}
\phit{r_m}{0} =(-1)^{r_m + 1} K^m_{r_m}
\end{equation} 
in~\eqref{eq:phit-characterization}, where~$K^m_{r_m}$ is the Legendre polynomial of degree~$r_m$ on~$I_m$. Therefore, revisiting~\eqref{eq:Leg2}, we observe that
\[
\Lnorm{\phit{r_m}{0}}{2}{I_m}
= \sqrtb{\frac{k_m}{2r_m+1}},
\]
which slightly improves the estimate from Lemma~\ref{lem:phiL2} above.
\end{remark}

The following result is the analog of Proposition~\ref{prop:phitimaxi} for the dG fundamental solution.

\begin{proposition}[Stability of $\psit{r_{m}}{\lambda}$]\label{prop:stab-psi}
Let~$\lambda>0$, and~$1\le m\le M$. For the dG fundamental solution from~\eqref{eq:phit-characterization} the identities
\begin{equation}\label{eq:psitimax}
 \Lnorm{\psit{r_{m}}{\lambda}}{\infty}{I_m} = \psit{r_{m}}{\lambda}(t_{m-1}),
\end{equation}
and
\begin{equation}\label{eq:psiti-der-max}
 \Lnorm{(\psit{r_m}{\lambda})'}{\infty}{I_m} = - (\psit{r_m}{\lambda})'(t_{m-1})
\end{equation}
hold true.
\end{proposition}

\begin{proof}
For simplicity of presentation, we suppose that~$r_m\ge 4$ (the cases~$0\le r_m\le 3$ can be verified directly). We show~\eqref{eq:psiti-der-max} first. For this purpose,
let us expand~$(\psit{r_{m}}{\lambda})'$ in a Legendre series, i.e.,
\begin{equation}\label{eq:expand1}
(\psit{r_{m}}{\lambda})' = \sum_{i=0}^{r_m - 1} b_i \hk{i},
\end{equation}
with coefficients~$b_0,\ldots b_{r_m-1}$. Recalling~\eqref{eq:intLeg}, and using~\eqref{eq:intLegprop}, for~$t\in I_m$, we have
 \begin{align}
\psit{r_{m}}{\lambda}(t) 
&= \psit{r_{m}}{\lambda}(t_{m-1})+\int_{t_{m-1}}^{t} (\psit{r_{m}}{\lambda})' \dd s \nonumber\\
&= \psit{r_{m}}{\lambda}(t_{m-1})+ \frac{k_m}{2}\sum_{i=0}^{r_m - 1} b_i Q^m_i (s) \dd s \nonumber \\
&= \psit{r_{m}}{\lambda}(t_{m-1})+\frac{k_m}{2}\left(b_{0}(\hk{0} + \hk{1}) + \sum_{i=1}^{r_m-1}  \frac{b_{i}}{2i+1}(\hk{i+1}-\hk{i-1})\right).\label{eq:expand2}
 \end{align}
Note that~$\hk{0}\equiv1$. Then, inserting~\eqref{eq:expand1} and~\eqref{eq:expand2} into~\eqref{eq:psi-characterization}, using the representation~\eqref{eq:liftingstrong} of the lifting operator, and comparing coefficients, leads to the equations
\begin{equation}\label{eq:coeff-der-psi}
\begin{split}
\lambda\psit{r_{m}}{\lambda}(t_{m-1})+\left(1+\frac{\lambda k_m}{2}\right)b_0 - \frac{\lambda k_m}{6}b_1 &=\frac{1}{k_m}\ez\\
\frac{\lambda k_m}{2} b_0 + b_1 - \frac{\lambda k_m}{10}b_2 &=-\frac{3}{k_m}\ez\\
\frac{\lambda k_m}{2(2i-1)}b_{i-1} + b_i - \frac{\lambda k_m}{2(2i+3)}b_{i+1}&=\frac{1}{k_m}(-1)^i(2i+1)\ez \quad (2\le i\le r_m-2)\\
\frac{\lambda k_m}{2(2r_m-3)} b_{r_m-2} + b_{r_m-1}&=\frac{1}{k_m}(-1)^{r_m-1}(2r_m-1)\ez \\
\frac{\lambda k_m}{2(2r_m-1)}b_{r_m-1}&=\frac{1}{k_m}(-1)^{r_m}(2r_m+1)\ez.
\end{split}
\end{equation}
Here, we denote by~$\ez=1-\psit{r_m}{\lambda}(t_{m-1})$ the error between the initial values of~$\psi$ from~\eqref{eq:slODE} and its dG approximation~$\psit{r_m}{\lambda}$. 
In order to show~\eqref{eq:psiti-der-max}, we first illustrate that the signs of the coefficients $b_0,\ldots,b_{r_m-1}$ are alternating. We focus on the case where $r_m$ is even. Let us first observe, by~\eqref{eq:init-val-psi} and~\eqref{eq:phitisandwich}, that
\begin{equation}\label{eq:psisandwich}
0<\psit{r_m}{\lambda}(t_{m-1})<1.
\end{equation} 
Rewriting the last equation in~\eqref{eq:coeff-der-psi}, we have
\begin{equation*}
b_{r_{m}-1} = \frac{(-1)^{r_m}}{2\lambda k_{m}^2}(4r_m^2-1)\ez.
\end{equation*}
Using~\eqref{eq:psisandwich}, we notice that
\begin{equation}\label{eq:ez+}
\ez>0,
\end{equation}
and because~$r_m$ is even, we arrive at~$b_{r_m-1}>0$. Then, from the second last equation in~\eqref{eq:coeff-der-psi}, we infer
\begin{equation*}
b_{r_m-2} = -\frac{2}{\lambda k_m}(2r_m-3) b_{r_m-1} + \frac{2(-1)^{r_m-1}}{\lambda k_{m}^2}(2r_m-3) (2r_m-1)\ez < 0.
\end{equation*}
Analogously, the third equation in~\eqref{eq:coeff-der-psi}, with~$i=r_m-2$, implies that
\begin{align*}
b_{r_{m}-3} &= - \frac{2}{\lambda k_m}(2r_m -5)b_{r_m - 2}\\
&\quad + \frac{2r_m -5}{2r_{m}-1}b_{r_m - 1} + \frac{2(-1)^{r_{m}-2}}{\lambda k_{m}^2}(2r_m -5)(2r_{m}-3)\ez > 0.
\end{align*}
We continue in the same way to conclude that~$\sgn(b_i)=(-1)^{i+1}$, for~$1\le i<r_{m}-1$. Finally, applying the second equation in~\eqref{eq:coeff-der-psi}, it holds that
\begin{equation*}
b_0 = \frac{2}{\lambda k_m}\left(-b_1+\frac{\lambda k_m}{10}b_2-\frac{3}{k_m}\ez\right) < 0.
\end{equation*}
Then, from~\eqref{eq:Leg1} and~\eqref{eq:Leg3}, we obtain
\begin{align*}
-(\psit{r_m}{\lambda})'(t_{m-1})
&=-\sum_{i=0}^{r_m-1}(-1)^ib_i
=\sum_{i=0}^{r_m-1}|b_i|
=\sum_{i=0}^{r_m-1}|b_i|\Lnorm{\hk{i}}{\infty}{I_m}\\
&\ge\Lnorm{(\psit{r_{m}}{\lambda})'}{\infty}{I_m},
\end{align*}
which gives~\eqref{eq:psiti-der-max}. For~$r_m$ odd we may proceed similarly.

In order to complete the proof, we show~\eqref{eq:psitimax}. To this end, we evaluate~\eqref{eq:psi-characterization} at $t=t_{m-1}$:
\begin{equation*}
(\psit{r_m}{\lambda})'(t_{m-1}) + \lambda \psit{r_m}{\lambda}(t_{m-1}) -\ez\lifting{r_m}{m}(t_{m-1}) = 0.
\end{equation*}
Since the coefficients of the lifting operator~$\lifting{r_m}{m}$ are alternating, and due to property~\eqref{eq:Leg1}, it is straightforward to see that~$\Lnorm{\lifting{r_m}{m}}{\infty}{I_m}=\lifting{r_m}{m}(t_{m-1})>0$. Hence, with~$\ez>0$ and by means of~\eqref{eq:psiti-der-max}, we see that
\[
 \lambda \psit{r_m}{\lambda}(t_{m-1}) = \Lnorm{(\psit{r_m}{\lambda})'}{\infty}{I_m} + \ez \Lnorm{\lifting{r_m}{m}}{\infty}{I_m}.
\]
Thus, in view of~\eqref{eq:psi-characterization}, which implies that
\[
\lambda\Lnorm{\psit{r_m}{\lambda}}{\infty}{I_m}
\le \Lnorm{(\psit{r_m}{\lambda})'}{\infty}{I_m} + \ez \Lnorm{\lifting{r_m}{m}}{\infty}{I_m},
\]
we conclude that $\psit{r_m}{\lambda}$ takes its maximum at~$t=t_{m-1}$.
\end{proof}

\subsection{Representation formulas}\label{sc:repform}
In this section we derive explicit representation formulas for the operator~$(\gaml{r_m})^{-1}$ defined in~\eqref{eq:gamscalar}. Observing~\eqref{eq:psi-characterization} and Lemma~\ref{lem:wr-plus-lifting}, it is sufficient to investigate how $\gamli$ acts on $\wr$.

\begin{lemma}\label{lem:gamli-on-wr}
Let $w\in \wr$, then it holds:
\begin{equation*}
 \gamli(w) = \int_{t_{m-1}}^{t}\ex{\lambda (s-t)} w(s) \dd s.
\end{equation*}
\end{lemma}

\begin{proof}
Let~$w\in\wr$, and choose $v \in \polrz$ with~$w=v'+\lambda v$. Then, we have the following equation:
\begin{align*}
 w(s)=\gaml{r_{m}}(v)(s) = v'(s) + \lambda v(s) = \ex{-\lambda (s-t_{m-1})} \der{s} \left(\ex{\lambda (s-t_{m-1})}v(s) \right),\quad s\in I_m.
\end{align*}
Hence, it follows that
$\der{s} \left(\ex{\lambda (s-t_{m-1})}v(s) \right)= \ex{\lambda (s-t_{m-1})} w(s)$.
Integrating with respect to $s$ over $(t_{m-1},t)$, and using~$v(t_{m-1})=0$, we obtain
\begin{equation*}
\ex{\lambda (t-t_{m-1})}v(t)= \int_{t_{m-1}}^{t}\ex{\lambda (s-t_{m-1})} w(s) \dd s,
\end{equation*}
and therefore,
\begin{equation*}
 (\gaml{r_m})^{-1}(w)(t) = v(t) = \int_{t_{m-1}}^{t}\ex{\lambda (s-t)} w(s) \dd s.
\end{equation*}
This completes the proof.
\end{proof}

\begin{proposition}\label{prop:gammainverse}
For any~$w\in\polr$ there holds
\begin{equation*}
 \gamli (w) = -\ex{-\lambda (t-t_{m-1})} \left(\int_{I_m}w\phit{r_m}{\lambda}\dd s\right)\etam(t) + \int_{t_{m-1}}^{t} \ex{\lambda (s-t)} w \dd s,
\end{equation*}
where
\begin{equation}\label{eq:etam}
\etam(t):=1 - \int_{t_{m-1}}^{t} \ex{\lambda (s-t_{m-1})} \lifting{r_m}{m}(1) \dd s.
\end{equation}
\end{proposition}

\begin{proof}
Consider any~$w\in\polr$. Then, Lemma \ref{lem:wr-plus-lifting} implies that there exist $\alpha \in \mathbb{R}$ and $w_0 \in \wr$ such that $w = w_{0} + \alpha \lifting{r_m}{m}(1)$. Hence, applying Lemma \ref{lem:gamli-on-wr}, and recalling~\eqref{eq:psi-characterization}, yields
\begin{align*}
 \gamli (w) &= \gamli (w_0) + \alpha \gamli (\lifting{r_m}{m}(1)) \nonumber \\
	    &= \int_{t_{m-1}}^{t}\ex{\lambda (s-t)} w_{0} \dd s + \alpha \psit{r_m}{\lambda} \nonumber \\
	    &= \int_{t_{m-1}}^{t}\ex{\lambda (s-t)} \left( w - \alpha \lifting{r_m}{m}(1) \right) \dd s + \alpha \psit{r_m}{\lambda}\\
	 &= \alpha \left(\psit{r_m}{\lambda} - \int_{t_{m-1}}^{t}\ex{\lambda (s-t)} \lifting{r_m}{m}(1) \dd s \right) + \int_{t_{m-1}}^{t}\ex{\lambda (s-t)} w \dd s.
\end{align*}
Setting
\begin{equation*}
 \Theta_{\lambda,m}^{r_m}:=\psit{r_m}{\lambda} - \int_{t_{m-1}}^{t}\ex{\lambda (s-t)} \lifting{r_m}{m}(1) \dd s,
\end{equation*}
and using the fact that $\psit{r_m}{\lambda}$ is the solution of \eqref{eq:psi-characterization}, an elementary calculation reveals that
\begin{equation*}
 (\Theta_{\lambda,m}^{r_m})'+ \lambda \Theta_{\lambda,m}^{r_m}= -\psit{r_m}{\lambda}(t_{m-1})\lifting{r_m}{m}(1).
\end{equation*}
Integrating this identity, we arrive at
\[
\Theta_{\lambda,m}^{r_m}(t) = \ex{-\lambda (t-t_{m-1})} \psit{r_m}{\lambda}(t_{m-1}) \etam(t).
\]
Therefore,
\[
 \gamli (w) = \alpha \ex{-\lambda (t-t_{m-1})} \psit{r_m}{\lambda}(t_{m-1}) \etam(t) + \int_{t_{m-1}}^{t} \ex{\lambda (s-t)} w \dd s.
\]
In order to determine the value of~$\alpha$, we employ Lemma~\ref{lem:phit-characterization} and~\ref{lem:init-val-psi}. This yields
\begin{equation*}
 \int_{I_m} w \phit{r_m}{\lambda} \dd t = \alpha\int_{I_m} \lifting{r_m}{m}(1)\phit{r_m}{\lambda} \dd t = \alpha \phit{r_m}{\lambda} (t_{m-1})
= -\alpha\psit{r_m}{\lambda}(t_{m-1}),
\end{equation*}
which directly leads to the desired formula.
\end{proof}

The following lemma gives an interesting interpretation of~$\etam$ defined in~\eqref{eq:etam}. Let us denote by
\begin{equation}\label{eq:error}
\errorode{m}{\lambda}:=\ex{-\lambda(t-t_{m-1})}-\psit{r_m}{\lambda}(t),\qquad t\in I_m,
\end{equation}
the pointwise error between the solution~$\psi$ of \eqref{eq:slODE}, and the dG fundamental solution~$\psit{r_m}{\lambda}$ from~\eqref{eq:psi-characterization2}.

\begin{lemma}\label{lem:etsformula}
We have the identity
\begin{equation*}
\ex{-\lambda (t-t_{m-1})}\etam(t)=\frac{\errorode{m}{\lambda}(t)}{\errorode{m}{\lambda}(t_{m-1})}, \qquad t\in I_m.
\end{equation*}
\end{lemma}

\begin{proof}
Due to~\eqref{eq:ez+}, let us first note that the right-hand side in the above identity is well-defined. Recalling~\eqref{eq:psi-characterization2}, and applying Proposition \ref{prop:gammainverse} with~$w=\lifting{r_m}{m}(1)$, we note that
\begin{align*}
\psit{r_m}{\lambda}(t) 
&= -\ex{-\lambda (t-t_{m-1})} \left(\int_{I_m}\lifting{r_m}{m}(1)\phit{r_m}{\lambda}\dd s\right)\etam(t) + \int_{t_{m-1}}^{t} \ex{\lambda (s-t)} \lifting{r_m}{m}(1) \dd s\\
&=- \phit{r_m}{\lambda}(t_{m-1})\ex{-\lambda (t-t_{m-1})} \etam(t) + \int_{t_{m-1}}^{t} \ex{\lambda (s-t)} \lifting{r_m}{m}(1) \dd s.
\end{align*}
By virtue of Lemma \ref{lem:init-val-psi}, this leads to
\begin{align*}
\psit{r_m}{\lambda}(t) 
&= \psit{r_m}{\lambda}(t_{m-1})\ex{-\lambda (t-t_{m-1})} + (1- \psit{r_m}{\lambda}(t_{m-1}))\int_{t_{m-1}}^{t} \ex{\lambda (s-t)} \lifting{r_m}{m}(1) \dd s,
\end{align*}
and thus,
\begin{align*}
-\errorode{m}{\lambda}(t)&=-\errorode{m}{\lambda}(t_{m-1})\ex{-\lambda (t-t_{m-1})} + \errorode{m}{\lambda}(t_{m-1})\int_{t_{m-1}}^{t} \ex{\lambda (s-t)} \lifting{r_m}{m}(1) \dd s\\
&=-\errorode{m}{\lambda}(t_{m-1})\ex{-\lambda (t-t_{m-1})}\etam(t).
\end{align*}
This shows the lemma.
\end{proof}

Summarizing the above results, we obtain the following representation expression.

\begin{corollary}\label{cor:gammainverse}
For any~$w\in\polr$, the identity
\begin{equation}\label{eq:gammainverse2}
 \gamli (w) = - \frac{\errorode{m}{\lambda}(t)}{\errorode{m}{\lambda}(t_{m-1})}\int_{I_m}w\phit{r_m}{\lambda}(t)\dd t + \int_{t_{m-1}}^{t} \ex{\lambda (s-t)} w \dd s
\end{equation}
holds true.
\end{corollary}

\subsection{Stability}\label{sc:stabilityscalar}

We are now in a position to derive stability bounds for~$\gamli$ as well as for the scalar dG time stepping solution from~\eqref{eq:nhdG}. In this section, let us suppose that~$\lambda>0$.

\begin{proposition}[$\fspace{L}^\infty$-$\fspace{L}^2$-Stability of $\gamli$]\label{prop:stabgamliL2}
Let $w \in \polr$, $1\le m\le M$. Then there holds the stability estimate
\begin{equation*}
\Lnorm{\gamli (w)}{\infty}{I_m} \leq C^{\fspace{L}^2}_{\lambda,r_m}k_m^{\nicefrac12} \Lnorm{w}{2}{I_m},
\end{equation*}
where
\begin{equation}\label{eq:CL2}
 C^{\fspace{L}^2}_{\lambda,r_m}:=\Cstab
 \left\lVert \frac{\errorode{m}{\lambda}}{\errorode{m}{\lambda}(t_{m-1})} \right \rVert_{\Lspace{\infty}{I_m}}  + 1.
\end{equation}
\end{proposition}

\begin{proof}
We separately bound the two terms on the right-hand side of~\eqref{eq:gammainverse2}. By means of the Cauchy-Schwarz inequality and Lemma~\ref{lem:phiL2}, we have
\begin{align*}
\left\lvert\int_{I_m}w\phit{r_m}{\lambda}\dd t \right\rvert 
&\le \Lnorm{w}{2}{I_m} \Lnorm{\phit{r_m}{\lambda}}{2}{I_m} 
\le k_m^{\nicefrac12}\Cstab\Lnorm{w}{2}{I_m}.
\end{align*}
Therefore, we infer that
\begin{equation}\label{eq:est-w11}
 \left\lvert \frac{\errorode{m}{\lambda}(t)}{\errorode{m}{\lambda}(t_{m-1})}\int_{I_m}w\phit{r_m}{\lambda}\dd t \right\rvert 
 \leq k_m^{\nicefrac12}\Cstab\Lnorm{\frac{\errorode{m}{\lambda}(t)}{\errorode{m}{\lambda}(t_{m-1})}}{\infty}{I_m}\lVert w \rVert_{L^{2}(I_m)}.
\end{equation}
Similarly, there holds
\begin{equation}\label{eq:est-w21}
\begin{split}
\left\|\int_{t_{m-1}}^te^{\lambda(s-t)}w(s)\dd s\right\|_{\fspace{L}^\infty(I_m)}
&\le\sup_{t\in I_m}\left(\int_{t_{m-1}}^te^{2\lambda(s-t)}\dd s\right)^{\nicefrac12}\|w\|_{\fspace{L}^2(I_m)}\\
&\le k_m^{\nicefrac12}\left\|w\right\|_{\fspace{L}^2(I_m)}.
\end{split}
\end{equation}
The two estimates~\eqref{eq:est-w11} and \eqref{eq:est-w21} immediately imply the asserted result.
\end{proof}

\begin{remark}($\fspace{L}^\infty$-$\fspace{L}^1$-Stability of $\gamli$)
As in the above Proposition~\ref{prop:stabgamliL2}, for $w \in \polr$, $1\le m\le M$, we can derive the bound
\begin{equation}\label{eq:L1stab}
\Lnorm{\gamli (w)}{\infty}{I_m} \leq C^{\fspace{L}^1}_{\lambda,r_m} \Lnorm{w}{1}{I_m},
\end{equation}
where
\begin{equation}\label{eq:CL1}
 C^{\fspace{L}^1}_{\lambda,r_m}:=
 \left\lVert \frac{\errorode{m}{\lambda}}{\errorode{m}{\lambda}(t_{m-1})} \right \rVert_{\Lspace{\infty}{I_m}}  + 1.
\end{equation}
Indeed, to see this, for the first term on the right-hand side of~\eqref{eq:gammainverse2}, we employ Proposition~\ref{prop:phitimaxi} to obtain
\begin{align*}
\left\lvert\int_{I_m}w\phit{r_m}{\lambda}\dd t \right\rvert 
&\le \Lnorm{w}{1}{I_m} \Lnorm{\phit{r_m}{\lambda}}{\infty}{I_m} 
= \Lnorm{w}{1}{I_m} \lvert {\phit{r_m}{\lambda}}(t_{m-1}) \rvert
\le \Lnorm{w}{1}{I_m}.
\end{align*}
Therefore, we obtain the bound
\begin{equation}\label{eq:est-w1}
 \left\lvert \frac{\errorode{m}{\lambda}(t)}{\errorode{m}{\lambda}(t_{m-1})}\int_{I_m}w\phit{r_m}{\lambda}\dd t \right\rvert \leq \Lnorm{\frac{\errorode{m}{\lambda}(t)}{\errorode{m}{\lambda}(t_{m-1})}}{\infty}{I_m}\lVert w \rVert_{L^{1}(I_m)}.
\end{equation}
As for the second term, we note that
\begin{equation}\label{eq:est-w2}
\begin{split}
\left\|\int_{t_{m-1}}^te^{\lambda(s-t)}w(s)\dd s\right\|_{\fspace{L}^\infty(I_m)}
&\le \sup_{t \in I_m}\int_{t_{m-1}}^te^{\lambda(s-t)}|w(s)|\dd s
\le\left\|w\right\|_{\fspace{L}^1(I_m)}.
\end{split}
\end{equation}
Thence, combining \eqref{eq:est-w1} and \eqref{eq:est-w2} gives~\eqref{eq:L1stab}.
\end{remark}

\begin{remark}\label{rem:errorode}
The term~$\left\lVert \errorode{m}{\lambda}(t_{m-1})^{-1}\errorode{m}{\lambda} \right \rVert_{\Lspace{\infty}{I_m}}$ arising in the constants~$C^{\fspace{L}^2}_{\lambda,r_m}$ and~$C^{\fspace{L}^1}_{\lambda,r_m}$ from~\eqref{eq:CL2} and~\eqref{eq:CL1}, respectively, can be estimated uniformly with respect to the time step~$k_m$ and the polynomial degree~$r_m$. In fact, performing an integration by parts in~\eqref{eq:etam}, we note that
\[
\etam(t)=1-\ex{\lambda(t-t_{m-1})}\rhoo(t) +\lambda \int_{t_{m-1}}^{t} \ex{\lambda (s-t_{m-1})} \rhoo(s) \dd s,
\]
where we define
\[
\rhoo(t):=\int_{t_{m-1}}^t\lifting{r_m}{m}(1)\dd s,\qquad t\in I_m.
\]
Rearranging terms, we obtain
\begin{align*}
\etam(t)=\ex{\lambda(t-t_{m-1})}(1-\rhoo(t)) -\lambda \int_{t_{m-1}}^{t} \ex{\lambda (s-t_{m-1})} (1-\rhoo(s)) \dd s,\qquad t\in I_m.
\end{align*}
Referring to~\cite[Lemma~1]{HolmWihler:15} it holds that
\begin{equation}\label{eq:rhomest}
\Lnorm{1-\rhoo}{\infty}{I_m}=1.
\end{equation} 
Consequently, we conclude that
\[
|\etam(t)|
\le\ex{\lambda(t-t_{m-1})}+\lambda \int_{t_{m-1}}^{t} \ex{\lambda (s-t_{m-1})}\dd s
= 2\ex{\lambda(t-t_{m-1})}-1.
\]
Recalling Lemma~\ref{lem:etsformula} results in 
\begin{equation}\label{eq:Cbound0}
\left\lVert \frac{\errorode{m}{\lambda}}{\errorode{m}{\lambda}(t_{m-1})} \right \rVert_{\Lspace{\infty}{I_m}}
\le \sup_{t\in I_m}\left(2-\ex{-\lambda(t-t_{m-1})}\right)=2-\ex{-\lambda k_m}.
\end{equation}
In particular, 
\begin{equation}\label{eq:Cbound}
1\le C^{\fspace{L}^2}_{\lambda,r_m}
\le\left(2-\ex{-\lambda k_m}\right)\Cstab  + 1,
\end{equation}
in~\eqref{eq:CL2}, and, thus, $C^{\fspace{L}^2}_{\lambda,r_m}\to 1$ as~$r\to\infty$ uniformly with respect to~$\lambda$.
Incidentally, a considerably more detailed analysis in~\cite{SchmutzDiss:18} reveals that there even holds $\left\lVert \errorode{m}{\lambda}(t_{m-1})^{-1}\errorode{m}{\lambda} \right \rVert_{\Lspace{\infty}{I_m}}= 1$. 
\end{remark}

\begin{remark}
For~$\lambda=0$, recalling~\eqref{eq:phit0}, we see that Proposition \ref{prop:gammainverse} implies a representation formula for $(\chi^{r_m})^{-1}$, cf.~\eqref{eq:chim}:
\begin{align*}
 (\chi_m^{r_m})^{-1} (w) 
 = (-1)^{r_{m}+1} \left(\int_{I_m}wK^m_{r_m}\dd t\right)\left(1 - \rhoo(t) \right) + \int_{t_{m-1}}^{t} w \dd s,
\end{align*}
for any~$w\in\polr$. Revisiting~\eqref{eq:rhomest}, and denoting by
\[
w_{r_m}:=\frac{2r_m+1}{k_m}\int_{I_m}wK^m_{r_m}\dd t
\]
the~$r_m$-th Legendre coefficient of~$w$, cf.~\eqref{eq:Leg2}, this leads to the stability estimate
\[
\Lnorm{(\chi_m^{r_m})^{-1} (w)}{\infty}{I_m}
\le |w_{r_m}|+\Lnorm{w}{1}{I_m},\qquad w\in\polr,
\]
which is an improvement of~\cite[Proposition~1]{HolmWihler:15}.
\end{remark}

The above Proposition~\ref{prop:stabgamliL2} immediately implies an $\Lspace{\infty}{I_m}$-stability bound for the dG time stepping solution~$U\in\polr$ from~\eqref{eq:nhdG}.

\begin{theorem}[$\fspace{L}^\infty$-stability of scalar dG solution]\label{prop:stab-dg-inhom-scal}
The dG solution~$U\in\polr$ from \eqref{eq:nhdG} satisfies
\begin{equation*}
\Lnorm{U}{\infty}{I_m}
\leq \lvert u_{m-1} \rvert+C^{\fspace{L}^2}_{\lambda,r_m}\Lnorm{f}{2}{I_m},
\end{equation*}
with~$C^{\fspace{L}^2}_{\lambda,r_m}$ from~\eqref{eq:CL2}.
\end{theorem}

\begin{proof}
Employing the triangle inequality to~\eqref{eq:nhdG}, together with the linearity of $\gamli$ and~$\lifting{r_m}{m}$, we have
\begin{equation*}
\Lnorm{U}{\infty}{I_m} 
\leq |u_{m-1}|\Lnorm{\gamli (\lifting{r_m}{m}(1))}{\infty}{I_m} +\Lnorm{\gamli (\projts f)}{\infty}{I_m}.
\end{equation*}
Recalling~\eqref{eq:psi-characterization2}, it follows that
\[
\Lnorm{U}{\infty}{I_m} 
\leq |u_{m-1}|\Lnorm{\psit{r_m}{m}}{\infty}{I_m} +\Lnorm{\gamli (\projts f)}{\infty}{I_m}.
\]
Using~\eqref{eq:psitimax} and~\eqref{eq:psisandwich}, and estimating the second term on the right-hand side of the above inequality by means of Proposition~\ref{prop:stabgamliL2}, we deduce that
\[
\Lnorm{U}{\infty}{I_m}
\leq \lvert u_{m-1} \rvert+C^{\fspace{L}^2}_{\lambda,r_m}k_m^{\nicefrac12}\Lnorm{\projts f}{2}{I_m}.
\]
The proof now follows from applying the $\fspace{L}^2(I_m)$-stability of~$\projts$.
\end{proof}

\section{Linear parabolic equations}\label{sc:abstract}

We now attend to the stability of the fully discrete dG time discretization~\eqref{eq:semilinear-dg-strong} for the linear parabolic evolution problem~\eqref{eq:slPDE}. For this purpose, for~$1\le m\le M$, we make use of the spectral decomposition of the discrete elliptic operator~$\operator{A}_m$ introduced in~\eqref{eq:Am}: Since $\operator{A}_m$ is self-adjoint and positive definite, there exist orthonormal basis functions~$\{ \varphi_i \}_{i=1}^{n_m}\subset\hspace{X}_m$, $\hspace{X}_m=\spn\{\varphi_1,\ldots,\varphi_{n_m}\}$,  which are eigenfunctions of~$\operator{A}_m$:
\begin{equation}\label{eq:Aspectral}
\iprod{\varphi_i}{\varphi_j}{\hspace H}=\delta_{ij},\qquad
\operator{A}_m\varphi_i=\lambda_i\varphi_i, \qquad i,j=1\ldots,n_m.
\end{equation}
Here, for~$1\le i\le n_m$, we signify by~$\lambda_i>0$ the (real) eigenvalue corresponding to~$\varphi_i$. Then, any function~$w\in\pol$ can be represented as
\begin{equation}\label{eq:wspectral}
w(t)=\sum_{i=1}^{n_m}a_i(t)\varphi_i,
\end{equation}
where~$a_i\in\polr$ are time-dependent coefficients, and there holds
\begin{equation*}
\|w(t)\|_{\hspace H}^2=\sum_{i=1}^{n_m} a_i(t)^2,\qquad t\in I_m.
\end{equation*}

\subsection{Stability of dG solution operator}

Following our approach in Section~\ref{sc:stabilityscalar} we now investigate the stability of the inverse of the discrete parabolic operator~$\gam{r_{m}}$ from~\eqref{eq:gam}.

\begin{proposition}\label{prop:gammapinverse}
Given $w \in \pol$, with a spectral representation as in \eqref{eq:wspectral}, then we have
\begin{equation}\label{eq:gamAi}
 \gami (w) = \sum_{i=1}^{n_m} \gamliindex{r_m}(a_i) \varphi_i,
\end{equation}
where~$\gam{r_{m}}$ and~$\gamlindex{r_m}$ are the discrete operators defined in~\eqref{eq:gam} and~\eqref{eq:gamscalar}, respectively. Moreover, the estimate
\begin{equation}\label{eq:Ginv}
\Lnorm{\gami (w)}{\infty}{I_m;\hspace{H}}\le C_mk_m^{\nicefrac12}\Lnorm{w}{2}{I_m;\hspace{H}}
\end{equation}
holds true, with
\begin{equation}\label{eq:CM}
C_m:=\max_{1\le i\le n_m}C^{\fspace{L}^2}_{\lambda_i,r_m},
\end{equation}
where~$C^{\fspace{L}^2}_{\lambda_i,r_m}$ is defined in~\eqref{eq:CL2}; cf. also~\eqref{eq:Cbound}.
\end{proposition}

\begin{proof}
Let~$w\in\pol$. Since~$\gam{r_m}$ is an isomorphism on~$\pol$ there exists a unique~$v\in\pol$, 
\[
v=\sum_{i=1}^{n_m}b_i\varphi_i,\qquad b_i\in\polr,\quad 1\le i\le n_m,
\]
such that~$w=\gam{r_m}(v)$. Equivalently, by linearity of~$\gam{r_m}$,
\begin{align*}
w&=\sum_{i=1}^{n_m}\gam{r_m}(b_i\varphi_i)
=\sum_{i=1}^{n_m}\chi^{r_m}_m(b_i)\varphi_i+b_i\operator{A}_m\varphi_i
=\sum_{i=1}^{n_m}\left(\chi^{r_m}_m(b_i)+\lambda_i b_i\right)\varphi_i\\
&=\sum_{i=1}^{n_m}\gamlindex{r_m}(b_i)\varphi_i.
\end{align*}
Comparing coefficients with~\eqref{eq:wspectral}, we infer that~$a_i=\gamlindex{r_m}(b_i)$, and thus, $b_i=\gamliindex{r_m}(a_i)$, for any~$i=1,\ldots,n_m$. Therefore,
\[
\gami(w)=v=\sum_{i=1}^{n_m}\gamliindex{r_m}(a_i)\varphi_i,
\]
which is~\eqref{eq:gamAi}. Now, employing~\eqref{eq:Aspectral}, we obtain
\[
\Lnorm{\gami(w)}{\infty}{I_m;\hspace{H}}^2
=\sup_{I_m}\sum_{i=1}^{n_m}\left|\gamliindex{r_m}(a_i)\right|^2
\le\sum_{i=1}^{n_m}\Lnorm{\gamliindex{r_m}(a_i)}{\infty}{I_m}^2.
\]
Applying Proposition~\ref{prop:stabgamliL2}, we arrive at
\[
\Lnorm{\gami(w)}{\infty}{I_m;\hspace{H}}
\le C_mk_m^{\nicefrac12}\left(\sum_{i=1}^{n_m}\Lnorm{a_i}{2}{I_m}^2\right)^{\nicefrac12}
=C_mk_m^{\nicefrac12}\Lnorm{w}{2}{I_m;\hspace{H}}.
\]
Recalling~\eqref{eq:Cbound} completes the proof.
\end{proof}

\subsection{Stability of homogeneous problem}

For~$1\le m\le M$, we denote by~$\hsol\in\pol$ the solution of the discrete problem
\begin{equation}\label{eq:hsol}
\frac{\dd}{\dd t}\hsol + \operator{A}_{m}\hsol 
+ \lifting{r_m}{m}(\hsol(t_{m-1})) 
=\lifting{r_m}{m}(\pi_mU_{m-1}^-),  
\end{equation}
where~$U_{m-1}^-\in\hspace{H}$ is a given value. Note that this is~\eqref{eq:semilinear-dg-strong} with~$f\equiv 0$.

\begin{lemma}\label{lem:properties-hsol}
Let $\hsol \in \pol$ be the solution of \eqref{eq:hsol}. 
Then, we have the stability estimate $\Lnorm{\hsol}{\infty}{I_m ; \hspace{H}} \leq \hnorm{U_{m-1}^{-}}{H}$.
\end{lemma}

\begin{proof}
We use the spectral decomposition $\pi_mU_{m-1}^-=\sum_{j=1}^{n_m}a_j\varphi_j$, with constant coefficients~$a_1,\ldots,a_{n_m}$.
Furthermore, exploiting the representation of the lifting operator from~\eqref{eq:liftingstrong}, and involving~\eqref{eq:psi-characterization}, there holds
\[
\lifting{r_m}{m}(\pi_mU_{m-1}^-)
=\sum_{j=1}^{n_m}a_j\lifting{r_m}{m}(1)\varphi_j
=\sum_{j=1}^{n_m}a_j\Gamma^{r_m}_{\lambda_j,m}(\psit{r_m}{\lambda_j})\varphi_j,
\]
where we slightly abuse notation by denoting the lifting operator on~$\hspace{X}_m$ and on~$\mathbb{R}$ in the same way. Hence, by virtue of~\eqref{eq:strong2}, with~$f\equiv 0$, and due to~\eqref{eq:gamAi}, we observe that
\begin{equation}\label{eq:hsolrep}
\hsol
=(\Gamma_m^{r_m})^{-1}(\lifting{r_m}{m}(\pi_mU_{m-1}^-))=\sum_{j=1}^{n_m} a_j\varphi_j\psit{r_m}{\lambda_j}.
\end{equation}
Using orthogonality, and applying~\eqref{eq:psitimax} and~\eqref{eq:psisandwich}, this leads to
\begin{align*}
\Lnorm{\hsol}{\infty}{I_m;\hspace{H}}^{2}
\le\sum_{j=1}^{n_m} a_j^2\Lnorm{\psit{r_m}{\lambda_j}}{\infty}{I_m}^2
\le\sum_{j=1}^{n_m}a_j^2 
=\|\pi_mU_{m-1}^-\|_{\hspace{H}}^2.
\end{align*}
Finally, applying the stability property \eqref{eq:stabpi} completes the proof.
\end{proof}

\begin{remark}\label{rm:errorhom}
We notice that~$\hsol$ defined in~\eqref{eq:hsol} is the fully discrete approximation of the solution of the homogeneous parabolic equation~\eqref{eq:slPDE}, with~$f\equiv 0$, on the time interval~$I_m$. For $t\in I_m$, the latter can be represented as~$\Psi(t)=e^{-\operator{A}(t-t_{m-1})}u(t_{m-1})$. Consequently, for~$t\in I_m$, the error satisfies the identity
\begin{equation}\label{eq:errorshom}
\begin{split}
\Psi(t)-\hsol(t)
&=e^{-\operator{A}(t-t_{m-1})}\left(u(t_{m-1})-\pi_mU_{m-1}^-\right)\\
&\quad+\left(e^{-\operator{A}(t-t_{m-1})}-e^{-\operator{A}_m(t-t_{m-1})}\right)\pi_mU_{m-1}^-\\
&\quad+\left(e^{-\operator{A}_m(t-t_{m-1})}\pi_mU^-_{m-1}-\hsol(t)\right).
\end{split}
\end{equation}
Let us briefly discuss the three terms on the right-hand side of the above equality. By stability, the first term in~\eqref{eq:errorshom} may simply be estimated by
\begin{align*}
\sup_{t\in I_m}&\hnorm{e^{-\operator{A}(t-t_{m-1})}\left(u(t_{m-1})-\pi_mU_{m-1}^-\right)}{\hspace{H}}\\
&\le \hnorm{u(t_{m-1})-\pi_mU_{m-1}^-}{\hspace{H}}
\le \hnorm{u(t_{m-1})-U_{m-1}^-}{\hspace{H}}+\hnorm{U_{m-1}^--\pi_mU_{m-1}^-}{\hspace{H}},
\end{align*}
which shows that this term is bounded by the error in the previous time step, and by a mesh change contribution. Moreover, the second term in~\eqref{eq:errorshom} refers to a Galerkin discretization error in space. Finally, using the spectral decomposition of~$\pi_mU_{m-1}^-$ as in the proof of Lemma~\ref{lem:properties-hsol}, and recalling~\eqref{eq:hsolrep}, the third term in~\eqref{eq:errorshom} can be written in the form
\[
e^{-\operator{A}_m(t-t_{m-1})}\pi_mU^-_{m-1}-\hsol(t)
=\sum_{j=1}^{n_m}a_j\varphi_j\left(e^{-\lambda_j(t-t_{m-1})}-\psit{r_m}{\lambda_j}\right),\qquad t\in I_m.
\]
Thus,
\[
\sup_{t\in I_m}\hnorm{\ex{-\operator{A}_{m}(t-t_{m-1})}\pi_mU_{m-1}^{-} - \hsol(t)}{\hspace{H}}^2
\leq \sum_{j=1}^{n_{m}} a_{j}^{2}\Lnorm{\errorode{m}{\lambda_{j}}}{\infty}{I_m}^2, 
\]
where the scalar error~$\errorode{m}{\lambda}$ is defined in~\eqref{eq:error}. Employing~\eqref{eq:Cbound0}, we notice that
$\Lnorm{\errorode{m}{\lambda_{j}}}{\infty}{I_{m}} \leq 2|\errorode{m}{\lambda_{j}}(t_{m-1})|$, and therefore obtain
\begin{align*}
\sup_{t\in I_m}\hnorm{\ex{-\operator{A}_{m}(t-t_{m-1})}\pi_mU_{m-1}^{-} - \hsol(t)}{\hspace{H}}
\le 2\hnorm{\pi_mU_{m-1}^-}{\hspace{H}}\sup_j\left|1-\psit{r_m}{\lambda_j}(t_{m-1})\right|.
\end{align*}
In particular, we see that the third term converges spectrally as~$r_m\to\infty$.
\end{remark}

\subsection{Stability of inhomogeneous problem}

Let us now turn to the stability of the fully discrete dG discretization~\eqref{eq:semilinear-dg}--\eqref{eq:u0} of the linear parabolic problem~\eqref{eq:slPDE}.

\begin{theorem}[$\fspace{L}^\infty(\hspace{H})$-stability of the dG time stepping method]\label{thm:main}
For any~$1\le m\le M$ the fully discrete dG time stepping solution~$U\in\prod_{m=1}^M\pol$ from~\eqref{eq:semilinear-dg} fulfills the stability estimate
\begin{equation}\label{eq:main}
\Lnorm{U}{\infty}{(0,t_m);\hspace{H}}
\le \|\pi_0u_0\|_{\hspace{H}}
+\gamma_mt^{\nicefrac12}_m\Lnorm{f}{2}{(0,t_m);\hspace{H}}.
\end{equation}
Here, we let $\gamma_m:=\max_{1\le i\le m}C_i$, where, for~$1\le i\le M$, the constant~$C_i$ is defined in~\eqref{eq:CM}.
\end{theorem}

\begin{proof}
For~$1\le i\le m$, we invert~\eqref{eq:strong2} to infer the solution formula
\begin{equation}\label{eq:soli}
U|_{I_i}=\gamii(\lifting{r_i}{i}(\pi_iU_{i-1}^-))+\gamii(\projtsi f)
=\hsoli+\gamii(\projtsi f),
\end{equation}
where~$\hsoli$ is the solution from~\eqref{eq:hsolrep}. Then, Lemma \ref{lem:properties-hsol} implies that
\[
\Lnorm{U}{\infty}{I_i;\hspace{H}}
\le\|U_{i-1}^-\|_{\hspace{H}}
+\Lnorm{(\Gamma_i^{r_i})^{-1}(\projtsi f)}{\infty}{I_i;\hspace{H}}.
\]
Furthermore, employing~\eqref{eq:Ginv} together with the $\fspace{L}^2(I_i;\hspace{H})$-stability of~$\projtsi$, we have
\[
\Lnorm{(\Gamma_i^{r_i})^{-1}(\projtsi f)}{\infty}{I_i;\hspace{H}}
\le C_ik_i^{\nicefrac12}\Lnorm{\projtsi f}{2}{I_i;\hspace{H}}
\le C_ik_i^{\nicefrac12}\Lnorm{f}{2}{I_i;\hspace{H}}.
\]
This yields the bound
\begin{equation}\label{eq:ULinf}
\Lnorm{U}{\infty}{I_i;\hspace{H}}
\le \|U_{i-1}^-\|_{\hspace{H}}
+C_ik_i^{\nicefrac12}\Lnorm{f}{2}{I_i;\hspace{H}}.
\end{equation}

Select now~$i^\star\in\{1,\ldots,m\}$ such that~$\Lnorm{U}{\infty}{(0,t_m);\hspace{H}}=\Lnorm{U}{\infty}{I_{i^\star};\hspace{H}}$. Then, with~\eqref{eq:ULinf} there holds
\[
\Lnorm{U}{\infty}{(0,t_m);\hspace{H}}
\le \|U_{i^\star-1}^-\|_{\hspace{H}}
+C_{i^\star}k_{i^\star}^{\nicefrac12}\Lnorm{f}{2}{I_{i^\star};\hspace{H}}.
\]
In order to estimate the first term on the right-hand side of the above inequality, we iterate the bound~\eqref{eq:ULinf}, thereby yielding
\begin{align*}
\Lnorm{U}{\infty}{(0,t_m);\hspace{H}}
&\le\Lnorm{U}{\infty}{I_{i^\star-1};\hspace{H}}
+C_{i^\star}k_{i^\star}^{\nicefrac12}\Lnorm{f}{2}{I_{i^\star};\hspace{H}}\\
&\le \|U_{i^\star-2}^-\|_{\hspace{H}}
+\sum_{i=i^\star-1}^{i^\star}C_{i}k_{i}^{\nicefrac12}\Lnorm{f}{2}{I_{i};\hspace{H}}\\[-2ex]
&\,\,\,\vdots\\[-2ex]
&\le \|U_{0}^-\|_{\hspace{H}}
+\sum_{i=1}^{i^\star}C_{i}k_{i}^{\nicefrac12}\Lnorm{f}{2}{I_{i};\hspace{H}}.
\end{align*}
Recalling~\eqref{eq:u0}, and applying the Cauchy-Schwarz inequality, we obtain
\[
\Lnorm{U}{\infty}{(0,t_m);\hspace{H}}
\le \|\pi_0u_{0}\|_{\hspace{H}}
+\gamma_m\left(\sum_{i=1}^{m}k_{i}\right)^{\nicefrac12}
\left(\sum_{i=1}^m\Lnorm{f}{2}{I_{i};\hspace{H}}^2\right)^{\nicefrac12},
\]
and the proof is complete.
%
\end{proof}

\begin{remark}
Noticing~\eqref{eq:Cbound}, we emphasize that the constant~$C_m$ appearing in~\eqref{eq:CM} tends to~$1$ as~$r_m\to\infty$. Consequently, $\gamma_m\to 1$, as~$r_m\to\infty$, in~\eqref{eq:main}.
\end{remark}

\begin{remark}
For~$t\in I_m$, the solution of the linear parabolic problem~\eqref{eq:slPDE} is given by 
\begin{equation*}
u(t) = \ex{-\operator{A}(t-t_{m-1})}u(t_{m-1}) + \int_{t_{m-1}}^{t}\ex{-\operator{A}(t-s)}f(s)\dd s.
\end{equation*}
Hence, recalling the solution formula~\eqref{eq:soli} for the discrete problem on~$I_m$, we have
\begin{align*}
u(t)-U(t)=\mathfrak{H}(t)+\mathfrak{I}(t),\qquad t\in I_m,
\end{align*}
where the terms $\mathfrak{H}(t)=\ex{-\operator{A}(t-t_{m-1})}u(t_{m-1})-\hsol(t)$, with~$\hsol$ from~\eqref{eq:hsolrep},  and
\[
\mathfrak{I}(t)=\int_{t_{m-1}}^{t}\ex{-\operator{A}(t-s)}f(s)\dd s-\gami(\projts f)(t)
\]
correspond to the homogeneous and inhomogeneous part of the PDE, respectively. Here,  to bound the error~$\Lnorm{u-U}{\infty}{I_m;\hspace{H}}$, we can employ our previous analysis in Remark~\ref{rm:errorhom} to control
$\Lnorm{\mathfrak{H}}{\infty}{I_m}$. Additionally, in order to estimate~$\Lnorm{\mathfrak{I}}{\infty}{I_m}$, let $\projts f=\sum_{i=1}^{n_m}{f}_i(t) \varphi_{i}$ be the spectral decomposition of~$\projts f$. By Proposition~\ref{prop:gammapinverse} and Corollary~\ref{cor:gammainverse} we have that~$\gami(\projts f)=\sum_{i=1}^{n_m}\gamliindex{r_m}(f_i)\varphi_i$, and thus,
\begin{align*}
\gami&(\projts f)\\
&=\sum_{i=1}^{n_m}- \frac{\errorode{m}{\lambda_{i}}(t)}{\errorode{m}{\lambda_{i}}(t_{m-1})}\int_{I_m}f_i(s)\phit{r_m}{\lambda_{i}}(s)\dd s 
+ \int_{t_{m-1}}^{t}\sum_{i=1}^{n_m} \ex{-\lambda_i (t-s)}f_i(s)\varphi_i  \dd s\\
&=\sum_{i=1}^{n_m}- \frac{\errorode{m}{\lambda_{i}}(t)}{\errorode{m}{\lambda_{i}}(t_{m-1})}\int_{I_m}f_i(s)\phit{r_m}{\lambda_{i}}(s)\dd s 
+ \int_{t_{m-1}}^{t} \ex{-\operator{A}_m(t-s)}\projts f(s)\dd s.
\end{align*}
Then,
\begin{align*}
\mathfrak{I}(t)
&=
\sum_{i=1}^{n_m}\frac{\errorode{m}{\lambda_{i}}(t)}{\errorode{m}{\lambda_{i}}(t_{m-1})}\int_{I_m}f_i(s)\phit{r_m}{\lambda_{i}}(s)\dd s
+\int_{t_{m-1}}^{t}\ex{-\operator{A}(t-s)}\left(f(s)-\projts f(s)\right)\dd s\\
&\quad+\int_{t_{m-1}}^{t}\left(\ex{-\operator{A}(t-s)}-\ex{-\operator{A}_m(t-s)}\right)\projts f(s)\dd s.
\end{align*}
We notice that the second integral is a data approximation term (which, with the aid of stability, can be estimated further), and the third integral relates to the spatial Galerkin discretization. Incidentally, the second term in~\eqref{eq:errorshom} and the third integral above add to the semi-discrete error in space; cf.~\cite[\S6]{Thomee:06}. Moreover, recalling~\eqref{eq:Cbound0}, the first term can be estimated by
\[
\left|\sum_{i=1}^{n_m}\frac{\errorode{m}{\lambda_{i}}(t)}{\errorode{m}{\lambda_{i}}(t_{m-1})}\int_{I_m}f_i(s)\phit{r_m}{\lambda_{i}}(s)\dd s\right|
\le 2\sum_{i=1}^{n_m}\left|\int_{I_m}f_i(s)\phit{r_m}{\lambda_{i}}(s)\dd s\right|.
\]
Even though both sides of the the above inequality are computable, we could proceed further by means of the Cauchy-Schwarz inequality (which results in a more pessimistic bound):
\[
\sum_{i=1}^{n_m}\left|\int_{I_m}f_i(s)\phit{r_m}{\lambda_{i}}(s)\dd s\right|
\le\left(\sum_{i=1}^{n_m}\Lnorm{f_i}{2}{I_m}^2\right)^{\nicefrac12}
\left(\sum_{i=1}^{n_m}\Lnorm{\phit{r_m}{\lambda_i}}{2}{I_m}^2\right)^{\nicefrac12}.
\]
Whilst the first term on the right-hand side of the above inequality can be bounded by~$\Lnorm{f}{2}{I_m;\hspace{H}}$ the second term can be estimated by means of Lemma~\ref{lem:phiL2}.
\end{remark}

\bibliographystyle{amsplain}
\bibliography{myrefs}

\providecommand{\bysame}{\leavevmode\hbox to3em{\hrulefill}\thinspace}
\providecommand{\MR}{\relax\ifhmode\unskip\space\fi MR }
\providecommand{\MRhref}[2]{%
  \href{http://www.ams.org/mathscinet-getitem?mr=#1}{#2}
}
\providecommand{\href}[2]{#2}
\begin{thebibliography}{10}

\bibitem{AkrivisMakridakis:04}
G.~Akrivis and C.~Makridakis, \emph{{Galerkin time-stepping methods for
  nonlinear parabolic equations}}, ESAIM: Mathematical Modelling and Numerical
  Analysis \textbf{38} (2004), no.~2, 261--289.

\bibitem{AkrivisMakridakisNochetto:09}
G.~Akrivis, C.~Makridakis, and R.~H. Nochetto, \emph{{Optimal order a
  posteriori error estimates for a class of {Runge}-{Kutta} and {Galerkin}
  methods}}, Numerische Mathematik \textbf{114} (2009), no.~1, 133--160.

\bibitem{AkrivisMakridakisNochetto:11}
\bysame, \emph{{Galerkin and {Runge}-{Kutta} methods: unified formulation, a
  posteriori error estimates and nodal superconvergence}}, Numerische
  Mathematik \textbf{118} (2011), no.~3, 429--456.

\bibitem{EJ-I}
K.~Eriksson and C.~Johnson, \emph{Adaptive finite element methods for parabolic
  problems. {I}. {A} linear model problem}, SIAM J. Numer. Anal. \textbf{28}
  (1991), no.~1, 43--77.

\bibitem{EJ-II}
\bysame, \emph{Adaptive finite element methods for parabolic problems. {II}.
  {O}ptimal error estimates in {$L_\infty L_2$} and {$L_\infty L_\infty$}},
  SIAM J. Numer. Anal. \textbf{32} (1995), no.~3, 706--740.

\bibitem{EJ-IV}
\bysame, \emph{Adaptive finite element methods for parabolic problems. {IV}.
  {N}onlinear problems}, SIAM J. Numer. Anal. \textbf{32} (1995), no.~6,
  1729--1749.

\bibitem{EJ-V}
\bysame, \emph{Adaptive finite element methods for parabolic problems. {V}.
  {L}ong-time integration}, SIAM J. Numer. Anal. \textbf{32} (1995), no.~6,
  1750--1763.

\bibitem{EJT85}
K.~Eriksson, C.~Johnson, and V.~Thom{\'e}e, \emph{Time discretization of
  parabolic problems by the discontinuous {G}alerkin method}, RAIRO Mod\'el.
  Math. Anal. Num\'er. \textbf{19} (1985), no.~4, 611--643.

\bibitem{ESV:16}
A.~Ern, I.~Smears, and M.~Vohral\'{i}k, \emph{Guaranteed, locally space-time
  efficient, and polynomial-degree robust a posteriori error estimates for
  high-order discretizations of parabolic problems}, in press in SIAM J. Numer.
  Analysis (2017).

\bibitem{GeorgoulisLakkisWihler:17}
E.~H. Georgoulis, O.~Lakkis, and T.~P. Wihler, \emph{A posteriori error bounds
  for fully-discrete $hp$-discontinuous {G}alerkin timestepping methods for
  parabolic problems}, Tech. Report 1708.05832, arxiv.org, 2017.

\bibitem{HolmWihler:15}
B.~Holm and T.~P. Wihler, \emph{Continuous and discontinuous {G}alerkin time
  stepping methods for nonlinear initial value problems with application to
  finite time blow-up}, in press in Numerische Mathematik (2017).

\bibitem{Jamet:78}
P.~Jamet, \emph{{Galerkin-type approximations which are discontinuous in time
  for parabolic equations in a variable domain}}, SIAM Journal on Numerical
  Analysis \textbf{15} (1978), no.~5, 912--928.

\bibitem{LakkisMakridakis:06}
O.~Lakkis and C.~Makridakis, \emph{Elliptic reconstruction and a posteriori
  error estimates for fully discrete linear parabolic problems}, Math. Comp.
  \textbf{75} (2006), no.~256, 1627--1658.

\bibitem{LarssonThomeeWahlbin:98}
S.~Larsson, V.~Thom{\'e}e, and L.~B. Wahlbin, \emph{{Numerical solution of
  parabolic integro-differential equations by the discontinuous {Galerkin}
  method}}, Math. Comp. \textbf{67} (1998), no.~221, 45--71.

\bibitem{MakridakisNochetto:03}
C.~Makridakis and R.~H. Nochetto, \emph{{Elliptic reconstruction and a
  posteriori error estimates for parabolic problems}}, SIAM Journal on
  Numerical Analysis \textbf{41} (2003), no.~4, 1585--1594.

\bibitem{MakridakisNochetto:06}
\bysame, \emph{{A posteriori error analysis for higher order dissipative
  methods for evolution problems}}, Numerische Mathematik \textbf{104} (2006),
  no.~4, 489--514.

\bibitem{MatacheSchwabWihler:05}
A.-M. Matache, C.~Schwab, and T.~P. Wihler, \emph{{Fast numerical solution of
  parabolic integrodifferential equations with applications in finance}}, SIAM
  Journal on Scientific Computing \textbf{27} (2005), no.~2, 369--393.

\bibitem{MatacheSchwabWihler:06}
A.-M. Matache, C.~Schwab, and T.P. Wihler, \emph{{Linear complexity solution of
  parabolic integro-differential equations}}, Numerische Mathematik
  \textbf{104} (2006), no.~1, 69--102. \MR{2232003}

\bibitem{Roubicek:05}
T.~Roub{\'{\i}}{\v{c}}ek, \emph{Nonlinear partial differential equations with
  applications}, International Series of Numerical Mathematics, vol. 153,
  Birkh\"auser/Springer Basel AG, Basel, 2005.

\bibitem{SchmutzDiss:18}
L.~Schmutz, \emph{Stability results for the d{G} time stepping method for
  parabolic evolution problems}, Ph.D. thesis, University of Bern, to appear.

\bibitem{SchotzauSchwab:00a}
D.~Sch{\"o}tzau and C.~Schwab, \emph{{An $hp$ a priori error analysis of the
  {DG} time-stepping method for initial value problems}}, Calcolo. A Quarterly
  on Numerical Analysis and Theory of Computation \textbf{37} (2000), no.~4,
  207--232. \MR{1812787}

\bibitem{SchotzauSchwab:00}
\bysame, \emph{{Time discretization of parabolic problems by the $hp$-version
  of the discontinuous {Galerkin} finite element method}}, SIAM Journal on
  Numerical Analysis \textbf{38} (2000), no.~3, 837--875. \MR{1781206}

\bibitem{SchotzauSchwab:01}
\bysame, \emph{{$hp$-discontinuous {Galerkin} time-stepping for parabolic
  problems}}, Comptes Rendus de l'Acad{\'e}mie des Sciences. S{\'e}rie I.
  Math{\'e}matique \textbf{333} (2001), no.~12, 1121--1126. \MR{1881245}

\bibitem{SchotzauWihler:10}
D.~Sch{\"o}tzau and Thomas~P. Wihler, \emph{{A posteriori error estimation for
  $hp$-version time-stepping methods for parabolic partial differential
  equations}}, Numer. Math. \textbf{115} (2010), no.~3, 475--509. \MR{2640055}

\bibitem{Thomee:06}
V.~Thom{\'e}e, \emph{{Galerkin finite element methods for parabolic problems}},
  second ed., Springer {Series} in {Computational} {Mathematics}, vol.~25,
  Springer-Verlag, Berlin, 2006. \MR{2249024}

\bibitem{PetersdorffSchwab:04}
T.~von Petersdorff and C.~Schwab, \emph{{Numerical solution of parabolic
  equations in high dimensions}}, M2AN. Mathematical Modelling and Numerical
  Analysis \textbf{38} (2004), no.~1, 93--127.

\bibitem{WerderGerdesSchotzauSchwab:01}
T.~Werder, K.~Gerdes, D.~Sch{\"o}tzau, and C.~Schwab, \emph{{$hp$-discontinuous
  {Galerkin} time stepping for parabolic problems}}, Computer Methods in
  Applied Mechanics and Engineering \textbf{190} (2001), no.~49-50, 6685--6708.
  \MR{1863353}

\end{thebibliography}
\end{document}